\documentclass[11pt,dvipsnames]{amsart}
\usepackage{euscript, oldgerm}           
\usepackage{amsmath,amsthm}     
\usepackage{amssymb}            
\usepackage[usenames,dvipsnames]{color}
\usepackage{graphicx}
\usepackage{hyperref}
\usepackage{verbatim} 
\usepackage{tikz-cd}

\newfont{\german}       {eufm10 at 12pt}

\colorlet{Cornflower}{gray!50!blue!78}

\newlength{\labwidth}
\newcommand{\labarrow}[1]{
\settowidth{\labwidth}{$\scriptstyle \;\; #1 \;\;$}
\stackrel{#1}{\smash{\hbox to \labwidth{\rightarrowfill}}
\vphantom{\longrightarrow}}
}

\newcommand{\mB} {{\mathcal B}}

\newcommand{\mC} {{\mathcal C}}
\newcommand{\mF} {{\mathcal F}}
\newcommand{\mH} {{\mathcal H}}

\newcommand{\mP} {{\mathcal P}}
\newcommand{\mR} {{\mathcal R}}
\newcommand{\bbF} {{\mathbb F}}

\newcommand{\bbZ} {{\mathbb Z}}
\newcommand{\bbO} {{\mathbb O}}
\newcommand{\bbC} {{\mathbb C}}
\newcommand{\bbH} {{\mathbb H}}
\newcommand{\bbone} {{\bf 1}}

\newcommand{\bbP} {{\mathbb P}}
\newcommand{\bbQ} {{\mathbb Q}}

\newcommand{\bbR} {{\mathbb R}}
\newcommand{\bbS} {{\mathbb S}}

\newcommand{\lra}{\longrightarrow}
\newcommand{\ra}{\rightarrow}

\newcommand{\Hi} {\mbox{{\scriptsize H}}}

\newtheorem{thm}{Theorem}[section]

\newtheorem{Lem}[thm]{Lemma}
\newtheorem{Cor}[thm]{Corollary}
\newtheorem{Thm}[thm]{Theorem}
\newtheorem{Pro}[thm] {Proposition}

\theoremstyle{definition}  

\newtheorem{example}[thm]{Example}
\newtheorem{Exa}[thm] {Example}

\newtheorem{Rem} [thm]{Remark}

\newcommand{\tensor} {\otimes}
\newcommand{\id} {\operatorname{id}}

\newcommand{\bmat}[1] {
  \begin{bmatrix}
    #1
  \end{bmatrix}}

\newcommand {\on}[1] {\operatorname{#1}}
\newcommand{\ind}{{\on{ind}}}

\newcommand {\ul}[1] {\underline{#1}}

\usepackage[frame,matrix,curve, arrow]{xy}

\xymatrixcolsep{1.9pc}                          
\xymatrixrowsep{1.9pc}
\newdir{ >}{{}*!/-5pt/\dir{>}}                  
\newdir{ |}{{}*!/-5pt/\dir{|}}                  

\begin{document}

\title{Codes, Vertex Operators and Topological Modular Forms}

\address{ 
School of Mathematics and Statistics, The University of Melbourne,
Parkville, Victoria, 3010, Australia\newline  
E-mail address: nganter@unimelb.edu.au
\newline \newline 
Fakult\"at f\"ur Mathematik,  Ruhr-Universit\"at Bochum, IB3/179,
D-44780 Bochum, Germany\newline E-mail address: gerd@laures.de} 



\date{\today}

\author{ Nora Ganter and Gerd Laures}
\begin{abstract} We describe a new link between the theory of
  topological modular forms and representations of vertex operator
  algebras obtained by certain lattices. The construction is motivated
  by the arithmetic Whitehead tower of the orthogonal groups. The
  tower discloses the role of codes in representation
  theory. \end{abstract}

\maketitle
\section{Introduction}

During the past decades there have been several attempts in various
directions to relate the theory of topological modular forms $tmf$ to
conformal field theories. Gourbonov, Malikov and Schechtman
\cite{MR1748287} among others   provided an interpretation of the
Witten genus in terms of chiral conformal field theories. More
precisely, they constructed a sheaf of differential vertex operator
algebras, known as the chiral de Rham complex, from which the elliptic
genus or the Witten genus can be recovered. This object might present
the Euler class in $tmf$ but a full geometric description of the
cohomology theory itself is still missing (cf.\cite{MR981378}\cite{MR2079378}). 
\par
This survey article has been written by algebraic topologists who are guided by the purpose to establish  a geometric interpretation of  the $tmf$-cocycles in terms of sheaves of vertex operator algebras. 
Its new perspective comes from coding theory which relates the
representation theory of very specific  lattice algebras to the
arithmetic Whitehead tower of the orthogonal groups. This connection
from  cocycle candidates to the string groups  differs from the usual
one, which only uses  index theory and the Witten genus.  \par 
Fortunately, the representation theory of vertex operator algebras
associated to even lattices is well understood. In this respect,
vertex operator algebras are at an advantage to the more
modern concept of factorization algebras, which is not part of this
work, though. 
It is well known that the representations  are related to modular
forms via the partition function. \par 
The paper contains a new result, Theorem \ref{Main}, establishing a one
to one correspondence between representations of the mentioned lattice
algebras and  the coefficients of $tmf(p)$, topological modular forms
with $\Gamma(p)$-structure. This result can be
interpreted as higher analogue of Atiyah-Bott-Shapiro's result which
relates representations of Clifford algebras to the coefficients
of $K$-theory.
\par
Codes already play an important
role in $K$-theory. This was explored in the work of Jay
A. Wood \cite{MR1022693}.
Phenomena like periodicity  and triality of the
spin groups can be best understood using coding theory. We use the (8,4)-Hamming code  to develop the symmetries and the representations of Clifford groups. These objects come together in the arithmetic Whitehead tower of the orthogonal groups. While climbing up the tower, we explain how vertex operator algebras become part of the picture  and how codes govern their representation theory. 

The article starts with a short introduction to coding theory and
$\theta$-series. The details can be found in \cite{MR2977354}. 
Section \ref{Lattices} does not contain new results but
formulates the van der 
Geer-Hirzebruch result for non linear codes.
This will be important
when it comes to partition functions of  vertex operator algebra
representations. Section \ref{AWT} deals with the arithmetic Whitehead
tower. It is shown how codes can be used to construct the interesting
representations of the Clifford groups. Afterwards, string extensions are
studied in the complex case and the role of lattices $L$ is
demonstrated.
In Section \ref{LVOA} we proceed to give a
model of the loop space extension of the torus $BL$ in terms of vertex
operator algebras and 
specify its representations. The main Theorem \ref{Main} gives a
one-to-one correspondence between  the representation ring  and the ring of
Hilbert modular forms. Finally, Section
\ref{section tmf} describes the relation of the theorem to topological
modular forms with level structures. This last section surveys the ideas to describe the
cocycles by 
sheaves of vertex algebra modules. Corollary \ref{Main3} is the
promised analogue of the Atiyah-Bott-Shapiro theorem. 

\subsection*{Acknowledgments} Throughout this project, the first author was supported by the
DFG Priority Programme 1786 and ARC Discovery Grant DP160104912. The second author
likes to thank the Max Planck Institute in Bonn for its
hospitality. He also benefited from a workshop at the Perimeter
Institute. The authors also like to thank G.\ H\"ohn, M.\ Kreck, A.\
Libgober,  A.\ Linshaw, L.\ Meier, V.\ Ozornova,  N.\ Scheithauer, B.\
Schuster for helpful conversations.

\section{Codes and Lattices}
\label{sec:codes}
Let $F$ be a field. 
 A {\em code} of length $n$ over $F$ is a subset of $F^n$. Let $C$ and
 $C'$ be codes over $F$  with the same length $n$. A (strict) morphism
 from $C$ to $C'$  is  
 a {\em monomial transformation}  $g$ such that 
  $g(C)\subset C'$. This means that $g$ is a linear  endomorphism of $F^n$ which in terms of the standard basis takes the form 
 $$g: e_i\mapsto c_i e_{\sigma^{-1}(i)}$$ 
 for some $c_i\in F^\times $ and some $\sigma \in \Sigma_n$.
 In particular, a strict automorphism  of $C$ is 
 an element in
 $$(F^\times )^n \rtimes \Sigma_n.$$
A code  is called {\em linear} if it is a linear subspace.
A linear code $C$ is {\em self-orthogonal} if $C\subseteq C^\perp$,
and {\em self-dual} if $C=C^\perp$ where orthogonality is with respect
to the standard inner product.
\subsection{The (8,4)-Hamming Code}\label{sec:Hamming}
The (8,4)-Hamming code $\mathcal H_8$ is a doubly even self-dual simply error
correcting binary code, and it is the smallest code with these properties.
As we shall see in Section \ref{sec:Spinors} below, these features are responsible for a particularly nice structure of
the real Clifford algebra $Cl_8$, which is at the heart of real Bott
periodicity, the $\widehat A$-genus and the triality
symmetries of the group $Spin(8)$. To recall the construction of
$\mH_8$, we introduce the Fano Plane
\[\mF\,=\, \bbP^2(\bbF_2),\]
viewed as a finite geometry, with lines given by the
2-dimensional subspaces of $\bbF_2^3$.
Its power set $\mP(\mF)$ endowed with the symmetric difference
operation is a seven dimensional $\bbF_2$-vector space. 
Inside it, we have the three dimensional linear subspace $\mC$ of line
complements.
For any vector $S\in\mP(\mF)$, the set-theoretic complement is
\[
  \mF\setminus S \, = \, \mF \, + \, S. 
\]
The {\em (7,4)-Hamming code} is the subspace
\[
  \mH_7\,=\,\mC\oplus\langle\mF\rangle
\]
of $\mP(\mF)$.
%
More precisely, any choice of numbering of the elements of the Fano plane defines an
identification $\bbF_2^7\cong\mP(\mF)$ with different choices giving
strictly isomorphic codes. We give two such numberings in Figure
\ref{fig:two_planes}.
Independent of the numbering, the incidence vector of the full set
$\mF$ is given by $\ul1:=(1,1,1,1,1,1,1)^T$.
The parity check isomorphism
\begin{eqnarray*}
  \bbF_2^7 & \xrightarrow{\,\,\,\cong\,\,\,} & (\bbF_2^8)^{ev}\\
  v & \,\,\longmapsto\,\, & (|v|,v) 
\end{eqnarray*}
transports $\mH_7$ into a doubly even code $\mH_8$ of length
eight. This is the (8,4)-Hamming Code.
The Fano plane is the smallest interesting example of a finite
projective geometry, and as such it is dually isomorphic to itself, meaning
that there is a bijection between points and lines preserving the
incidence relation. We can realize this duality in a manner that
decomposes $(\bbF_2^7)^{ev}$ as the sum of two (strictly isomorphic)
copies of $\mC$.
To do so, we number the points of $\mF$ in two different ways, as
depicted in Figure \ref{fig:two_planes}.
\begin{figure}
  \centering
\begin{tikzpicture}[scale=2.4]
  \tikzset{->-/.style={decoration={
  markings,
  mark=at position #1 with {\arrow{>}}},postaction={decorate}}}
  \node at (1,0.577) [anchor=center,circle, draw, fill=SeaGreen,minimum
  height=.6cm,name=a4] {$\mathbf 5$};
  \draw (a4) [line width=1mm, SeaGreen] circle (.577);
  \node at (1,1.732) [circle, draw, fill=SeaGreen, minimum height=.6cm,name=a7] {{$\mathbf 7$}};
  \node at (.5,0.866) [circle, draw, fill=SeaGreen,  minimum height=.6cm,name=a1] {$\mathbf 1$};
  \node at (1.5,0.866) [anchor=center,circle, draw, fill=SeaGreen, minimum height=.6cm,name=a2] {$\mathbf 2$};
  \node at (0,0) [circle, draw, anchor=center, fill=SeaGreen, minimum height=.6cm,name=a6] {$\mathbf 3$};
  \node at (2,0) [circle, draw, fill=SeaGreen, minimum height=.6cm,name=a5] {$\mathbf 6$};
  \node at (1,0) [circle, draw, fill=SeaGreen, minimum height=.6cm,name=a3] {$\mathbf 4$};

  \draw [SeaGreen, line width=1mm] (a1) -- node [midway] {} (a7);
  \draw [SeaGreen, line width=1mm] (a2) -- node [midway] {} (a4);
    \draw [SeaGreen, line width=1mm] (a4) -- node [midway] {} (a6);
  \draw [SeaGreen, line width=1mm] (a6) -- node [midway] {} (a1);
  \draw [SeaGreen, line width=1mm] (a5) -- node [midway] {} (a3);
  \draw [SeaGreen, line width=1mm] (a3) -- node [midway] {} (a6);
  \draw [SeaGreen, line width=1mm] (a1) -- node [midway] {} (a4);
  \draw [SeaGreen, line width=1mm] (a4) -- node [midway] {} (a5);
  \draw [SeaGreen, line width=1mm,] (a7) -- node [midway] {} (a2);
  \draw [SeaGreen, line width=1mm,] (a2) -- node [midway] {} (a5);
  \draw [SeaGreen, line width=1mm] (a3) -- node [midway] {} (a4);
  \draw [SeaGreen, line width=1mm] (a4) -- node [midway] {} (a7);
  \draw [SeaGreen, line width=1mm] (a1) -- node [midway] {} (a7);
  \draw [SeaGreen, line width=1mm] (a1) -- node [midway] {} (a7);
  \draw [SeaGreen, line width=1mm] (a1) -- node [midway] {} (a7);
  \draw [SeaGreen, line width=1mm] (a1) -- node [midway] {} (a7);

  \node at (2.5,0){};
\end{tikzpicture}
\begin{tikzpicture}[scale=2.4]
  \tikzset{->-/.style={decoration={
  markings,
  mark=at position #1 with {\arrow{>}}},postaction={decorate}}}
  \node at (1,0.577) [anchor=center,circle, draw, fill=Cornflower,minimum
  height=.6cm,name=a4] {$\mathbf 4$};
  \draw (a4) [line width=1mm, Cornflower] circle (.577);
  \node at (1,1.732) [circle, draw, fill=Cornflower, minimum
  height=.6cm,name=a7] {$\mathbf 7$};
  \node at (.5,0.866) [circle, draw, fill=Cornflower,  minimum height=.6cm,name=a1] {$\mathbf 1$};
  \node at (1.5,0.866) [anchor=center,circle, draw, fill=Cornflower, minimum height=.6cm,name=a2] {$\mathbf 2$};
  \node at (0,0) [circle, draw, anchor=center, fill=Cornflower, minimum height=.6cm,name=a6] {$\mathbf 6$};
  \node at (2,0) [circle, draw, fill=Cornflower, minimum height=.6cm,name=a5] {$\mathbf 5$};
  \node at (1,0) [circle, draw, fill=Cornflower, minimum height=.6cm,name=a3] {$\mathbf 3$};

  \draw [Cornflower, line width=1mm] (a1) -- node [midway] {} (a7);
  \draw [Cornflower, line width=1mm] (a2) -- node [midway] {} (a4);
    \draw [Cornflower, line width=1mm] (a4) -- node [midway] {} (a6);
  \draw [Cornflower, line width=1mm] (a6) -- node [midway] {} (a1);
  \draw [Cornflower, line width=1mm] (a5) -- node [midway] {} (a3);
  \draw [Cornflower, line width=1mm] (a3) -- node [midway] {} (a6);
  \draw [Cornflower, line width=1mm] (a1) -- node [midway] {} (a4);
  \draw [Cornflower, line width=1mm] (a4) -- node [midway] {} (a5);
  \draw [Cornflower, line width=1mm,] (a7) -- node [midway] {} (a2);
  \draw [Cornflower, line width=1mm,] (a2) -- node [midway] {} (a5);
  \draw [Cornflower, line width=1mm] (a3) -- node [midway] {} (a4);
  \draw [Cornflower, line width=1mm] (a4) -- node [midway] {} (a7);
  \draw [Cornflower, line width=1mm] (a1) -- node [midway] {} (a7);
  \draw [Cornflower, line width=1mm] (a1) -- node [midway] {} (a7);
  \draw [Cornflower, line width=1mm] (a1) -- node [midway] {} (a7);
  \draw [Cornflower, line width=1mm] (a1) -- node [midway] {} (a7);

\end{tikzpicture}
\caption{Two numberings of the Fano plane}
    \label{fig:two_planes}
\end{figure}
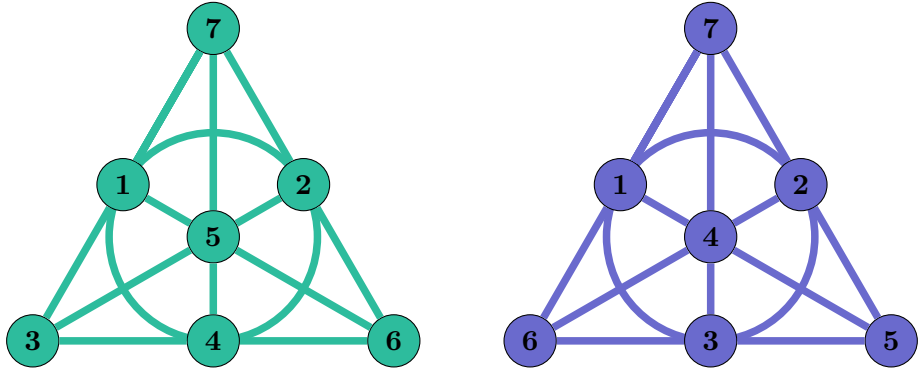
The dual isomorphism is then obtained by numbering the lines on either
side as in Figure \ref{fig:incidence}.
\begin{figure}[h]
  \centering
\begin{tikzpicture}[scale=1,every node/.style={minimum size=1.1cm-\pgflinewidth, outer sep=0pt}]
  \draw[very thick](-1.8,2.6)--(6.6,2.6)--(6.6,-5.8)--(-1.8,-5.8)--(-1.8,2.6);

  \draw[very thick](-1.8,1.4)--(6.6,1.4);
  \draw[very thick](-1.8,.2)--(6.6,.2);
  \draw[very thick](-1.8,-1)--(6.6,-1);
  \draw[very thick](-1.8,-2.2)--(6.6,-2.2);
  \draw[very thick](-1.8,-3.4)--(6.6,-3.4);
  \draw[very thick](-1.8,-4.6)--(6.6,-4.6);
    
  \draw[very thick](-.6,2.6)--(-.6,-5.8);
  \draw[very thick](.6,2.6)-- (.6,-5.8);
  \draw[very thick](1.8,2.6)--(1.8,-5.8);
  \draw[very thick](3,2.6)--(3,-5.8);
  \draw[very thick](4.2,2.6)--(4.2,-5.8);
  \draw[very thick](5.4,2.6)--(5.4,-5.8);

  \node at (-3,2) {Line 1};
  \node at (-3,.8) {Line 2};
  \node at (-3,-.4) {Line 3};
  \node at (-3,-1.6) {Line 4};
  \node at (-3,-2.8) {Line 5};
  \node at (-3,-4) {Line 6};
  \node at (-3,-5.2) {Line 7};
  \node at (-1.2,3) {1};
  \node at (0,3) {2};
  \node at (1.2,3) {3};
  \node at (2.4,3) {4};
  \node at (3.6,3) {5};
  \node at (4.8,3) {6};
  \node at (6,3) {7};

    \node[] at (-2,+2) {};
    \node[fill=Cornflower] at (0,2) {};
    \node[fill=SeaGreen] at (1.2,2) {};
    \node[fill=SeaGreen] at (2.4,2) {};
    \node[fill=Cornflower] at (3.6,+2) {};
    \node[fill=SeaGreen] at (4.8,+2) {};
    \node[fill=Cornflower] at (+6,+2) {};
    \node[fill=SeaGreen] at (-1.2,+0.8) {};
    \node[fill=Cornflower] at (+1.2,+.8) {};
    \node[fill=Cornflower] at (+2.4,+.8) {};
    \node[fill=SeaGreen] at (+3.6,+.8) {};
    \node[fill=SeaGreen] at (+4.8,+.8) {};
    \node[fill=Cornflower] at (+6,+.8) {};
    \node[fill=Cornflower] at (-1.2,-.4) {};
    \node[fill=SeaGreen] at (0,-.4) {};
    \node[fill=Cornflower] at (+2.4,-.4) {};
    \node[fill=Cornflower] at (+3.6,-.4) {};
    \node[fill=SeaGreen] at (+4.8,-.4) {};
    \node[fill=SeaGreen] at (+6,-.4) {};
    \node[fill=Cornflower] at (-1.2,-1.6) {};
    \node[fill=SeaGreen] at (0,-1.6) {};
    \node[fill=SeaGreen] at (+1.2,-1.6) {};
    \node[fill=SeaGreen] at (+3.6,-1.6) {};
    \node[fill=Cornflower] at (+4.8,-1.6) {};
    \node[fill=Cornflower] at (+6,-1.6) {};
        \node[fill=SeaGreen] at (-1.2,-2.8) {};
    \node[fill=Cornflower] at (0,-2.8) {};
    \node[fill=SeaGreen] at (+1.2,-2.8) {};
    \node[fill=Cornflower] at (+2.4,-2.8) {};
    \node[fill=Cornflower] at (+4.8,-2.8) {};
    \node[fill=SeaGreen] at (+6,-2.8) {};
    \node[fill=Cornflower] at (-1.2,-4) {};
    \node[fill=Cornflower] at (0,-4) {};
    \node[fill=Cornflower] at (+1.2,-4) {};
    \node[fill=SeaGreen] at (+2.4,-4) {};
    \node[fill=SeaGreen] at (+3.6,-4) {};
    \node[fill=SeaGreen] at (+6,-4) {};
    \node[fill=SeaGreen] at (-1.2,-5.2) {};
    \node[fill=SeaGreen] at (0,-5.2) {};
    \node[fill=Cornflower] at (+1.2,-5.2) {};
    \node[fill=SeaGreen] at (+2.4,-5.2) {};
    \node[fill=Cornflower] at (+3.6,-5.2) {};
    \node[fill=Cornflower] at (+4.8,-5.2) {};
\end{tikzpicture}
\caption{Incidence table relating points and lines in
  Figure \ref{fig:two_planes}. The dual isomorphism is reflected in
  the diagonal symmetry.} 
\label{fig:incidence}
\end{figure}
Writing
$c_i$ for the complement of Line $i$ in the right-hand
picture and $b_i$ for the complement of Line $i$ in the left-hand
picture, we have
\[
  c_i \,\,+ \,\, b_i  \,=\, \{i\}\,\,+\,\, \ul1.
\]
As a result, we obtain the decompositions
\[
  \left(\bbF_2^7\right)^{ev} \,=\, \mB\,\oplus\, \mC
\]
and
\[
  \left(\bbF_2^8\right)^{ev}\cong \bbF_2^7 \,=\,\mB\,\oplus\,\langle\ul1\rangle\oplus
  \mC,
\]
where $\mB$ and $\mC$ are the subspaces of line complements consisting
of the $b_i$ (respectively $c_i$) and the zero vector.
A further feature of point-line duality is that lines on the right govern addition in
$\mB$ and vice versa:
three distinct numbers $i$, $j$ and $k$ form a line in
the right-hand picture, if and only if
\[
  b_i\,+\,b_j\,+\,b_k\,=\,0.
\]
Similarly, 
\[
  c_i\,+\,c_j\,+\,c_k\,=\,0
\]
in $\mC$ if and only if $i$, $j$ and $k$ form a line on the left.
This will become important in our explicit calculations of spinor representations.
\subsection{The Standard Ternary Code}\label{sec:LT}
 A code is called {\em ternary} if $F=\bbF_3$. 
 A famous ternary code is
 $$ {\mathcal T}=\{ (s,a,a+s,a+2s)|\, a,s,\in \bbF_3\}.$$
 It is a linear code of dimension 2 in $\bbF_3^4$ which is self-dual. Its automorphism group is $GL_2(\bbF_3)$ as one easily verifies. This group coincides
 with the automorphism group of the Lubin-Tate curve  
$$C: \; y^2+y=x^3$$ 
 over $\bbF_4=\{0,1,\omega, \overline{\omega}\}$. This curve plays an
 important role in the construction of the spectrum $tmf$. It will be reconsidered in Section \ref{section tmf}. 
\subsection{The ternary Golay Code}
 The ternary Golay code is a self-dual linear code $C$ in
 $\bbF_3^{12}$ with 729 words.   Its
 automorphism group is the Mathieu group $2.M_{12}$. This code seems to be related to   the 
 triality of spin groups as considered in \ref{section triality}, but the authors 
 have not yet been able to make this connection precise.
\section{Lattices over integers of number fields}\label{Lattices}
For details the reader is referred to Chapters 5.1 to 5.4 of
\cite{MR2977354}.
 Let $p$ be an odd prime.  Let $F$ be the number field  $\bbQ(\zeta)$
 for some primitive $p$th root of unity. This is a field extension
 over $\bbQ$  of degree $p-1$.
 For $r$ ranging from $1$ to $p-1$ there are
 embeddings $\sigma_r$  of $F$ in $\bbC$ given by 
 $\sigma_r(\zeta)=\zeta^r$.
 The {\em trace} of an element $\alpha\in F$ is defined by
 $$\mbox{Tr}(\alpha)=\sum_{r=1}^{p-1}\sigma_r(\alpha ).$$
 We write $\bbO$ for the ring of integers in $F$. The underlying
 additive group of $\mathbb O$ is free abelian of rank 
$p-1$ generated by $\zeta^0, \ldots, \zeta^{p-2}$. 
Let $\bbP$ be the principal ideal of $\bbO$ generated by $1-\zeta$.
Then $\mathbb P$ is the kernel of the map $\rho: \bbO \ra  \bbZ /p\bbZ$ given by 
\begin{eqnarray}
 (a_0+a_1\zeta +\cdots +a_{p-2}\zeta^{p-2}) \mapsto  (a_0+a_1 +\cdots +a_{p-2}).
\end{eqnarray}
The pairing
$$ \langle x, y\rangle= \mbox{Tr}\left(\frac{x\bar{y}}{p}\right)$$
defines a bilinear form on $\bbO$ with values in the rationals.  It makes $\bbP$ into an even lattice, that is, 
 $ \langle x,y\rangle \in \bbZ \mbox{ and } \langle x,x\rangle \in 2 \bbZ.$
It is isomorphic to the root lattice $A_{p-1}.$
\subsection{Lattices from Codes} 
More generally, let $\rho: \bbO^n \ra \bbF_p^n$ be the reduction map modulo the principal ideal $\bbP$ in each coordinate. For a code  $C$  set
$$\Gamma_C=\rho^{-1} (C)\subset \bbO^n .$$
Assume that  $C$ is a self-orthogonal linear code. Then  the symmetric bilinear form
\begin{eqnarray}\label{bil}
\langle x, y\rangle & = &
\sum_{i=1}^{n} \mbox{Tr} \left(
\frac{x_i \bar{y}_i}{p}
\right)
 .\end{eqnarray}
turns  $\Gamma_C$ to an even lattice of rank $n(p-1)$
with  discriminant  $p^{n-2m}$. The dual lattice 
$$\Gamma^\vee = \mbox{Hom}(\Gamma,\bbZ )=\{x\in  F^n|\; \langle x, y\rangle \in \bbZ\}$$
 satisfies
 \begin{eqnarray}\label{dual}
 \Gamma_C^\vee \cong \Gamma_{C^\perp}
 \end{eqnarray}
(see f.i.\ \cite[Proposition 5.2 and Lemme 5.5]{MR2977354} for details).
In particular, if $C$ is self dual, then $\Gamma_C$ is isomorphic to
its dual lattice. Self-dual lattices are known as {\em unimodular} lattices.
\begin{example}
The standard ternary code gives the root lattice $E_8$. The ternary
Golay code produces  an even unimodular lattice of rank 
24 which contains $12A_2$.
\end{example}
\subsection{Theta series}
Lattices come with theta series, and these are key to elliptic genera.
The usual theta series of the lattice $\Gamma_C$ is
$$ \vartheta_C(z)=\sum_{x\in \Gamma_C} e^{\pi i z \mbox{\small Tr} \left(
\frac{x \bar{x}}{p}
\right)}
=\sum_{x\in \Gamma_C} e^{2 \pi i z \mbox{\small Tr}_k \left(
\frac{x \bar{x}}{p}
\right)
}.$$ 
where $k=\bbQ(\zeta +\zeta^{-1})$ is the real subfield of $F$. There
is an refinement of $\vartheta_C$  which is a holomorphic function on
$r={(p-1)/2}$-copies of the upper half-plane $\bbH$: set 
$$\theta_C(z) = \sum_{x\in \Gamma_C} e^{2 \pi i \mbox{\small Tr}_k \left(z
\frac{x \bar{x}}{p}
\right)} \mbox{ with } \mbox{\small Tr}_k \left(z
\frac{x \bar{x}}{p}
\right)= \sum_{l=1}^{r} z_l \frac{\sigma_l(x\bar{x})}{p}.
$$
Set 
$ \theta_j=\theta_{\bbP +j}$ for  $j\in \bbF_p$.
These functions are symmetric Hilbert modular forms of weight $m=1$
with respect to the congruence group  $$\Gamma({\textgoth
  p})=\{A=\left(\begin{array}{cc}a&b\\c&d\end{array}\right)\in
Sl_2(\bbO)| \, A=
\left(\begin{array}{cc}1&0\\0&1\end{array}\right)\mbox{ mod
}{\textgoth p}\}  $$ where ${\textgoth p}=\bbP\cap k$. In detail, this
means they are holomorphic functions $f$ which transform as  
$$f(\sigma_1(A)z_1,\ldots,
\sigma_{r}(A)z_{r})=f(z)\prod_{l=1}^{r}(\sigma_l(c)z_l+\sigma_l(d))^m.$$  
The symmetry means that they are invariant under the action of the
Galois group $\mbox{Gal}(k,\bbQ ) $ which permutes the coordinates of
$\bbH^r$. 
It turns out that every function $\theta_C$ is a sum of monomials in
the $\theta_j$s and hence is a symmetric Hilbert modular form.  
In order to formulate this result more precisely, recall that 
the weight enumerator of a code $C$ is the polynomial
$$W_C(x_0,x_1, \ldots x_r)=\sum_{w\in C} x_0^{l_0(w)} x_1^{l_1(w)}\cdots x_r^{l_r(w)} $$
where $l_0(w)$ is the number of zeroes in $w$ and $l_j(w)$ is the
number of $\pm j$. The following result is the Theorem of Alpbach by
van der Geer-Hirzebruch. It is formulated there for linear
self-orthogonal codes in \cite{MR2977354}[Theorem 5.3]. We repeat the
proof below because the assumption is unnecessary. 
\begin{thm}\label{GeerHirz}
Let $C$ be arbitrary (not necessary linear). Then
$$\theta_C=W_C(\theta_0,\theta_1, \ldots, \theta_r)$$
In particular, each code of length $n$ gives a symmetric $\Gamma({\textgoth p})$-Hilbert modular form of weight $n$.
\end{thm}
\begin{proof}  For $c\in C$ calculate
\begin{eqnarray*}
\sum_{x\in \rho^{-1}(c)}e^{2\pi i  \mbox{\small Tr}_k \left(z
\frac{x_i \bar{x}_i}{p}
\right)}
&=&\!\!\!\!\!\!\! 
\sum_{x_1\in \rho^{-1}(c_1)}e^{2\pi i   \mbox{\small Tr}_k \left(z
\frac{x_1 \bar{x}_1}{p}
\right)}\ldots \sum_{x_n\in \rho^{-1}(c_n )}e^{2\pi i   \mbox{\small Tr}_k \left(z
\frac{x_n \bar{x}_n}{p}
\right)} \\[+3ex]
&=& \theta^{l_0(c)}_0\theta_1^{l_1(c)}\cdots \theta_{r}^{l_r(c)}(z)
\end{eqnarray*}
since 
$$\theta_j(z)=\sum_{x\in \bbP+j}e^{2\pi i   \mbox{\small Tr}_k \left(z
\frac{x \bar{x}}{p}\right)}=\sum_{x\in \bbP-j}e^{2\pi i   \mbox{\small Tr}_k \left(z
\frac{x \bar{x}}{p}\right)}.
$$
The claim follows after summing over all code words.
\end{proof}
\begin{Cor}\label{Sl2}
 The usual theta series $\vartheta_C$ of $C$ is a modular form for $\Gamma(p)$ of weight $nr$. Self dual codes give $Sl_2(\bbZ)$-invariant modular forms.
\end{Cor}
\begin{proof}
The first statement follows from the theorem after taking the diagonal map from $\bbH$ to  $\bbH^r$. The second statement uses (\ref{dual}) and the well known fact that $\theta$-series of self dual lattices are invariant under the full modular group. It can also be obtained from the theorem by analyzing the action of $Sl_2(\bbF_p)$ (see \cite[Section 5]{MR2977354} for $p=3,5$).
\end{proof}
\begin{Rem}\label{DN}
Doi-Naganuma have constructed  a map from the group of modular forms for the congruence group $\Gamma_0(p)$ to Hilbert modular forms for $\Gamma({\textgoth p})$ which preserves the weight and Hecke eigenforms. It would be interesting to analyze those symmetric Hilbert modular forms which are in the image of the DN lift in terms of code words.

\end{Rem}
\section{Arithmetic Whitehead towers}\label{AWT}
The organizing principle for the different cohomology theories
considered in this paper comes from the Whitehead towers of the
orthogonal and the unitary groups. These are the towers
$$O(n) \longleftarrow SO(n)  \xleftarrow{\,\,\bbZ/2\,\,}
Spin(n)  \xleftarrow{\,\,PU(\Hi )\,\,}  String(n)
\longleftarrow \ldots$$
and 
$$U(n) \longleftarrow SU(n)  
\xleftarrow{\,\,PU(\Hi )\,\,}  String(n)
\longleftarrow \ldots$$
defined for large $n$. The step to the right is obtained by killing off the
lowest homotopy group. For instance, $SO(n)$ is a connected
component of $O(n) $ and $Spin(n) $ is its simply connected
cover, which turns out to be 2-connected. For $n\geq 3$ the smallest non-trivial
homotopy group of $Spin(n)$ is
\[
  \pi_3(Spin(n)) \cong
  \begin{cases}
    \mathbb Z \cong \langle Spin(3)\rangle & \text{ for $n\neq4$}\\
    \mathbb Z\oplus\mathbb Z & \text{ for $n=4$},
  \end{cases}
\]
\cite[Chapter VI, Theorem 4.17]{MimuraToda91}.

The next step in the tower, $String(n)$, is
classified by a generator of the integral cohomology
$H^4(BSpin(n))$. To be specific, we recall the low degree
cohomology groups:
\begin{enumerate}
\item
Let  $n\geq3$.  Then $H^k(BSpin(n))$ vanishes for
$k\leq3$ and is cyclic of infinite order for $k=4$ except for $n=4$
where it is two copies of the integers.  
\item For $n\geq 4$, one generator is $p_1/2$ where $p_1$ is the first
  Pontryagin class of the canonical bundle associated to the adjoint
  representation.  
\item For $n=3$,  the Euler class of the spinor representation
  generates. It coincides with $p_1/4$.  
\end{enumerate}
An explicit model realizing $String(n)$ as topological group
can be found in \cite[Section 5]{MR2079378}. This method can be
continued, but for the purposes of the paper at hand, we are only
interested in the first steps of the tower.
For $n\geq 4$, $String(n)$ fibers over $Spin(n)$ with
fibre $PU(\mbox{H})=K(\bbZ,2)$, the projective unitary group of an infinite
dimensional separable Hilbert space. There is also a realization of
$String_n(\bbR)$ as compact 2-group with fibre $\mathbb B U(1)$, whose
geometric realization gives the topological group, see
for instance \cite{Schommer-Pries11}.
\subsection{The Clifford groups and codes}
We revisit the $(8,4)$-Hamming Code in the context of Whitehead tower.
If we make the identifications
$$\bbF_2^8\,\cong\, O(1)^8\,\cong\, O(8)\cap \on{Diag}$$
then the Fano line complements can be realized as the the three-fold 
tensor products of the diagonal Pauli matrices
\[
  \sigma_0 = \bmat{1&0\\0&1}\quad\quad\text{  and  }\quad\quad   \sigma_3 = \bmat{1&0\\0&-1}.
\]
%
More precisely, in the second picture in Figure \ref{fig:two_planes}, 
the point $[a\negmedspace:\negmedspace b\negmedspace:\negmedspace c]\in \mF$ has
number $4a+2b+c$. The resulting isomorphism $$\mP(\mF)\cong\bbF_2^7\cong(\bbF_2^8)^{ev}$$
 identifies the linear map 
\begin{eqnarray*}
  C\negmedspace : \bbF_2^3 &\,\longrightarrow\, & \mP(\mF)\\
  p & \longmapsto & \mF\setminus p^\perp
\end{eqnarray*}
with the map
\begin{eqnarray*}
   \bbF_2^3 &\,\longrightarrow\, & O(1)^8\\
    (a,b,c) & \longmapsto & \sigma_3^a\tensor \sigma_3^b\tensor \sigma_3^c.
\end{eqnarray*}
If $p$ is a point in $\mathcal F$ (a line in $\mathbb F_2^3$),
then the orthogonal complement $p^\perp$ is a line in $\mF$ (a plane
in $\mathbb F_2^3$), and $\mF\setminus p^\perp$ is a line complement.
So, the image of
$C$ is the subspace of line complements.
The vector $\ul1$ corresponds to $-\id$, so the the full Hamming
code $\mH_8$ is identified with the subgroup
\[
  \mH = \left\{ \pm \sigma_3^a\tensor \sigma_3^b\tensor \sigma_3^c\mid
  a,b,c\in \{0,1\}\right\}.
\]
The various properties making $\mH_8$ suitable for coding theory are,
to a topologists, statements about the Whitehead tower.
The pre-image of $O(1)^n$ inside $Pin(n)$ is the
extraspecial 2-group $F_n$ of Clifford words. Writing $-1\in F_n$ for the
central element, this group consists of words in the symbols
$e_0,\dots e_{n-1}$ subject to the relations
\[
  e_i^2 = -1\quad\text{and}\quad e_ie_j=-e_je_i.
\]
The code $\mH_8$ is even, meaning \[\mH\subset SO(8).\]
Moreover, $\mH_8$ is self dual, so its spin-cover
  \[\widetilde \mH\subset  Spin(8)\] is abelian.
The fact that $\mH_8$ is doubly even makes $\widetilde\mH$ elementary
abelian, 
\[
  \widetilde \mH \cong \langle\pm 1\rangle\times \mH,
\]
meaning that $\mH$ lifts to a subgroup of $Spin(8)$.
Indeed, $\widetilde \mH$ is a maximal abelian subgroup of the
group $F_8^{ev}$ of even Clifford words. 
The fact that $\mH_8$ is simply error correcting means that
the Clifford symbols $e_0,\dots e_7$ form a system of coset
representatives for $F_8^{ev}/\widetilde \mH$. The orthogonal decomposition
\[
  \bbF_8^{ev}\,=\, \mB\oplus \mH_8
\]
discussed in Section \ref{sec:Hamming} allows us to write $F_8^{ev}$ as
semi-direct product of two elementary abelian groups,
\[
  F_8^{ev} \,\cong \, \mB \ltimes \widetilde\mH.
\]
Here $\mB$ acts on $\widetilde \mH$ by
\[
  b\negmedspace : \widetilde h \longmapsto (-1)^{\langle b,h\rangle}
  \,\widetilde h,
\]
where $\widetilde h\in \widetilde\mH$ is a lift of $h\in \mH$.
\subsection{Spinors}\label{sec:Spinors}
The twisted group algebra of $F_n$ is the Clifford algebra $Cl_n$ of $\bbR^n$ with its
standard quadratic form 
$$ Cl_n= \bbR[F_n]/(-1)_\bbR \sim (-1)_{F_n}.$$
In particular, Clifford modules coincide with
centre-faithful representations $\rho$ of $F_n$.
Because of the inclusion
\[
  Spin(n) \subset Cl_n^{ev} \cong \bbR^{tw}[F_n^{ev}]
\]
\cite{MR0167985}, it follows that representations of
$F_n^{ev}$ give rise to representations of the spin groups.
Let now $n=8$. The spinor representations $\Delta^+$ and $\Delta^-$
are the two irreducible modules over the Clifford algebra $Cl_8$. Each
has dimension $8$. 
Following Pressley and Segal \cite[Section 12.6]{PressleySegal86}, the
spinor representations are constructed by inducing up centre-faithful
representations of a maximal abelian subgroup. This is the point
where it becomes essential that our maximal abelian subgroup is, in
fact, elementary abelian, i.e., that $\mH_8$ is a Type II code. Let
$\chi^+$ and $\chi^-$ be the characters
\begin{eqnarray*}
  \chi^+\negmedspace : \langle \pm 1\rangle \times \mH_8 &
                                                           \,\longrightarrow
                                                           \, & O(1)\\
  -1 & \longmapsto & -1\\
  h& \longmapsto & 1
\end{eqnarray*}
and
\begin{eqnarray*}
  \chi^-\negmedspace : \langle \pm 1\rangle \times \mH_8 &
                                                           \,\longrightarrow
                                                           \, & O(1)\\
  -1 & \longmapsto & -1\\
  h& \longmapsto & h_0.
\end{eqnarray*}
Then the eight dimensional spinor representations $\Delta^{\pm}$ are
given as 
\[\Delta^\pm=\ind_{\widetilde \mH}^{F_8^{ev}}\,\chi^\pm.\]
Due to the semi-direct product decomposition
$F_8^{ev}=\mB\ltimes\widetilde\mH$, elements of $\mB$ act as
permutation matrices of order two, which turn out to be three-fold
tensor products of the $2\times 2$ permutation matrices
\[
  \sigma_0 = \bmat{1&0\\0&1}\quad\text{and}\quad   \sigma_1 = \bmat{0&1\\1&0}.
\]
More precisely, $\mB$ acts by addition on itself, which is governed by
$\mC$ and given by the doubly even permutations
\[
\begin{array}{ccc}
  b_1&\, \longmapsto\,& (01)\,(23)\,(45)\,(67) = \sigma_0\tensor\sigma_0\tensor\sigma_1, \\[+.6em]
  b_2&\, \longmapsto\,& (02)\,(13)\,(46)\,(57) = \sigma_0\tensor\sigma_1\tensor\sigma_0,\\[+.6em] 
  b_3&\, \longmapsto\,& (03)\,(12)\,(47)\,(56) = \sigma_0\tensor\sigma_1\tensor\sigma_1, \\[+.6em]
  b_4&\, \longmapsto\,& (04)\,(15)\,(26)\,(37) = \sigma_1\tensor\sigma_0\tensor\sigma_0, \\[+.6em]
  b_5&\, \longmapsto\,& (05)\,(14)\,(27)\,(36) = \sigma_1\tensor\sigma_0\tensor\sigma_1, \\[+.6em]
  b_6&\, \longmapsto\,& (06)\,(17)\,(24)\,(35) = \sigma_1\tensor\sigma_1\tensor\sigma_0, \\[+.6em]
  b_7&\, \longmapsto\,& (07)\,(16)\,(25)\,(34) = \sigma_1\tensor\sigma_1\tensor\sigma_1. 
\end{array}
\]
The subgroup $\mH$ is fixed by $\Delta^+$ and acted upon by
$\chi^-\tensor\id$ for $\Delta^-$.
\subsection{Triality}\label{section triality}
The relationship between spinor groups and codes was already explored
in the work of Wood \cite{MR1022693}, who also writes about the triality symmetry.
Together with the fundamental representation
\[
  \pi: Spin(8)\longrightarrow SO(8),
\]
the spinor representations $\Delta^\pm$ form the eight dimensional
irreducible representations of $Spin(8)$. 
All three are double covers, their respective kernels generated by the three 
non-trivial center elements:
\[
  \ker(\Delta^+)\,=\,\langle\omega\rangle,\quad\quad
  \ker(\Delta^-)\,=\,\langle-\omega\rangle,\quad\quad
  \ker(\pi)\,=\,\langle-1\rangle,
\]
where $\omega=e_0\cdots e_7$ is the second generator of the center of
$Spin(8)$. The spin covers of the spinor representations $\Delta^\pm$
are the extraordinary outer automorphisms of $Spin(8)$
depicted in Figure \ref{fig:triality}.
\begin{figure}[h]
  \centering
  \[
    \begin{tikzpicture}[scale=1.5]
      \node at (0,0) [name=so] {$SO(8)$};
      \node at (-1,1) [name=d+] {$Spin(8)$};
      \node at (1,1) [name=d-] {$Spin(8)$};
      \node at (0,-1) [name=pi] {$Spin(8)$};
      \draw[->,thick] (d+) -- node [midway, left] {$\Delta^+\,\,$} (so);
      \draw[->,thick] (d-) -- node [midway, right] {$\Delta^-$} (so);
      \draw[->,thick] (pi) -- node [midway, right] {$\pi$} (so);
      \draw[->,thick,dashed,blue!80!green] (d+) to [bend right=60] node
      [midway,left] {$\textcolor{red}{(-1,\omega)}\quad\quad\boldsymbol{\tau^+}$} (pi);
      \draw[->,thick,dashed,blue!80!green] (d-) to [bend left=60] node
                                 [midway,right]
                                {$\boldsymbol{\tau^-}\quad\quad \textcolor{red}{(-1,\omega,-\omega)}$} 
      (pi);
      \draw[<->,thick,dashed,dashed,blue!80!green] (d+) to [bend left=50] node
      [midway,above,text width=1.5cm, align=center]
      {
        $\textcolor{red}{(\omega,-\omega)}$\\[+1em] $\boldsymbol\sigma$} (d-);      
    \end{tikzpicture}
  \]
  \caption{Triality symmetry of the Dynkin diagram $D_4$. The
    extraordinary automorphisms $\tau^\pm$ are the lifts of the spinor
  representations, while $\sigma$ is conjugation with $e_0$. The
  action of the symmetric group $S_3$ on the center
  is indicated in red.}
  \label{fig:triality}
\end{figure}
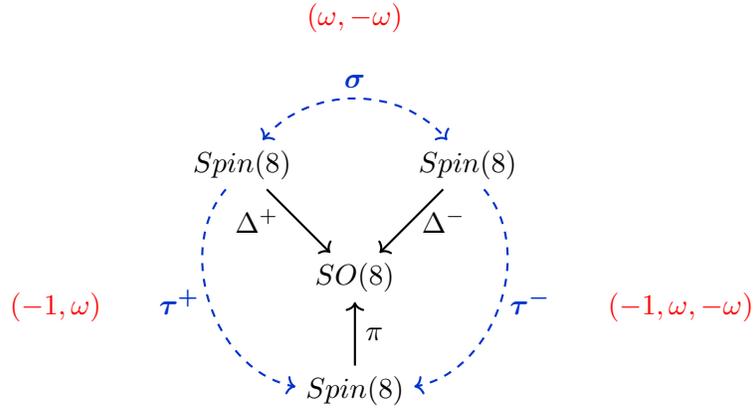

%
\subsection{Bott periodicity}
The same objects that define triality are also responsible for real
Bott periodicity. 
The Bott element is given by the full spinor representation 
\[
  \Delta = \ind_{\widetilde \mH}^{F_8}\,\chi^+ = \Delta^+\oplus e_0\Delta^+, 
\]
combining $\Delta^+$ and $\Delta^-$ in the sense that
\[\on{res}_{F_8^{ev}}\Delta \cong \Delta^+\oplus\Delta^-.\]
The Pauli matrix point of view has the advantage that it gives an
explicit description of the Bott element, which is more direct than
the approach in the standard literature \cite{MR0167985}
\cite{LawsonMichelsohn89}.
One checks directly that the matrices
\[ E_1= {\small \left[\textcolor{gray!50}{\begin{array}{rr|rr|rr|rr}
 & \textcolor{black}{-1} &  &  &  &  &  &  \\
\textcolor{black}{1} &  &  &  &  &  &  &  \\
\hline
  &  &  & \textcolor{black}{-1} &  &  &  &  \\
 &  & \textcolor{black}{1} &  &  &  &  &  \\
\hline
  &  &  &  &  & \textcolor{black}{1} &  &  \\
 &  &  &  & \textcolor{black}{-1} &  &  &  \\
\hline
  &  &  &  &  &  &  & \textcolor{black}{1} \\
 &  &  &  &  &  & \textcolor{black}{-1} & 
\end{array}}\right]
\quad
E_2 = \left[\textcolor{gray!50}{\begin{array}{rr|rr|rr|rr}
 &  & \textcolor{black}{-1} &  &  &  &  &  \\
 &  &  & \textcolor{black}{1} &  &  &  &  \\
\hline
 \textcolor{black}{1} &  &  &  &  &  &  &  \\
 & \textcolor{black}{-1} &  &  &  &  &  &  \\
\hline
  &  &  &  &  &  & \textcolor{black}{-1} &  \\
 &  &  &  &  &  &  & \textcolor{black}{1} \\
\hline
  &  &  &  & \textcolor{black}{1} &  &  &  \\
 &  &  &  &  & \textcolor{black}{-1} &  & 
      \end{array}}\right]}
\]

\[
  E_3 ={\small \left[\textcolor{gray!50}{\begin{array}{rr|rr|rr|rr}
 &  &  & \textcolor{black}{-1} &  &  &  &  \\
 &  & \textcolor{black}{-1} &  &  &  &  &  \\
\hline
  & \textcolor{black}{1} &  &  &  &  &  &  \\
\textcolor{black}{1} &  &  &  &  &  &  &  \\
\hline
  &  &  &  &  &  &  & \textcolor{black}{-1} \\
 &  &  &  &  &  & \textcolor{black}{-1} &  \\
\hline
  &  &  &  &  & \textcolor{black}{1} &  &  \\
 &  &  &  & \textcolor{black}{1} &  &  & 
              \end{array}}\right]
            \quad
E_4 = \left[\textcolor{gray!50}{\begin{array}{rr|rr|rr|rr}
 &  &  &  & \textcolor{black}{-1} &  &  &  \\
 &  &  &  &  & \textcolor{black}{-1} &  &  \\
\hline
  &  &  &  &  &  & \textcolor{black}{1} &  \\
 &  &  &  &  &  &  & \textcolor{black}{1} \\
\hline
\textcolor{black}{1} &  &  &  &  &  &  &  \\
 & \textcolor{black}{1} &  &  &  &  &  &  \\
\hline
  &  & \textcolor{black}{-1} &  &  &  &  &  \\
 &  &  & \textcolor{black}{-1} &  &  &  & 
      \end{array}}\right]}
\]  

\[
  E_5 = {\small \left[\textcolor{gray!50}{\begin{array}{rr|rr|rr|rr}
     &  &  &  &  & \textcolor{black}{-1} &  &  \\
     &  &  &  & \textcolor{black}{1} &  &  &  \\
  \hline
     &  &  &  &  &  &  & \textcolor{black}{-1} \\
     &  &  &  &  &  & \textcolor{black}{1} &  \\
  \hline
     & \textcolor{black}{-1} &  &  &  &  &  &  \\
   \textcolor{black}{1} &  &  &  &  &  &  &  \\
  \hline
     &  &  & \textcolor{black}{-1} &  &  &  &  \\
     &  & \textcolor{black}{1} &  &  &  &  & 
             \end{array}}\right]
              \quad
E_6 =  \left[\textcolor{gray!50}{
\begin{array}{rr|rr|rr|rr}
 &  &  &  &  &  & \textcolor{black}{-1} &  \\
 &  &  &  &  &  &  & \textcolor{black}{-1} \\
\hline
  &  &  &  & \textcolor{black}{-1} &  &  &  \\
 &  &  &  &  & \textcolor{black}{-1} &  &  \\
\hline
  &  & \textcolor{black}{1} &  &  &  &  &  \\
 &  &  & \textcolor{black}{1} &  &  &  &  \\
\hline
 \textcolor{black}{1} &  &  &  &  &  &  &  \\
 & \textcolor{black}{1} &  &  &  &  &  & 
\end{array}}\right]}            
\]

\[
 E_7 ={\small \left[
\textcolor{gray!50}{  \begin{array}{rr|rr|rr|rr}
   &  &  &  &  &  &  & \textcolor{black}{-1} \\
   &  &  &  &  &  & \textcolor{black}{1} &  \\
  \hline
   &  &  &  &  & \textcolor{black}{1} &  &  \\
   &  &  &  & \textcolor{black}{-1} &  &  &  \\
  \hline
   &  &  & \textcolor{black}{1} &  &  &  &  \\
   &  & \textcolor{black}{-1} &  &  &  &  &  \\
  \hline
   & \textcolor{black}{-1} &  &  &  &  &  &  \\
  \textcolor{black}{1} &  &  &  &  &  &  & 
                         \end{array}}
                     \right]}
\]
satisfy the Clifford relations. Hence \[\Delta^{\pm}(e_0e_i) = \pm E_i\]
defines two modules over the even Clifford algebra, which are readily
checked to have the desired effect on centers. Bott periodicity is the
statement that the representation
\[
  \Delta: Cl_8  \longrightarrow  Mat_{16\times 16} 
\]
is an isomorphism. This is now an easy consequence of the fact that
the real Pauli matrices
\[
  \sigma_0, \quad \sigma_1,\quad \sigma_1\sigma_0,\,\,\,\,\text{and}\,\,\,\, \sigma_3.
\]
form a basis of the $2\times 2$-matrices, taking into account that
$\Delta$ sends the 256 positive Clifford words to 
the 256 possible 4-fold tensor products of the real Pauli matrices.
Indeed, the 64 three-fold tensor products are the possible products of the
matrices $E_1,\dots,E_7$ above, which are easily read off from the Pauli
matrix realizations of $\mH$ and $\mB$. Combined with
$\omega\mapsto \sigma_3\tensor\id\tensor\id\tensor\id$ and
$e_1\mapsto \sigma_1\tensor\id\tensor\id\tensor\id$, these determine
all of $\Delta$.

If we allow ourselves to work over the complex numbers, the
doubly even condition can be dropped, and an even self-dual code will
suffice, since unlike $O(1)$, the group $U(1)$ can accommodate elements
of order 4.
This is the reason why complex Bott periodicity already occurs at
$n=2$, where we have the first even self-dual code.
\subsection{The integral spin and string groups.}
There is a version of the Whitehead tower for finite
groups, where the successive killing of lower degree group cohomology
has the meaning of forming first the commutator subgroup then the
universal central extension (Schur extension), then 
the categorical Schur extension.
This process requires the commutator subgroup of the original group to
be perfect, so we will limit ourselves to considering only such groups.
For details we refer the reader to \cite{EpaGanter17}.
As before, a tower of topological groups can be obtained by taking classifying
spaces of the fibers, so that the categorification interpretation is
really a matter of taste.

We saw above how codes fit into the Whitehead tower for
$O(n)=O_n(\bbR)$. Since codes, however, are finite objects, a more
natural home is the arithmetic Whitehead tower for 
$O_n(\bbZ)$, the group of orthogonal matrices with integral
entries.
These group are known as the {\em hyper-octahedral groups}.
They are generated by permutation matrices and diagonal matrices 
with $\pm 1$. It is not hard to verify the
isomorphism 
\begin{equation}\label{rtimes}
O_n(\bbZ)\cong \Sigma_n\ltimes \bbF_2^n.
\end{equation}
The spinor representations $\Delta^\pm$ discussed above
map the group $F_8^{ev}$ to $SO_8(\bbZ)$, or rather its spin double
cover, which we will denote $Spin_8(\bbZ)$.\par 
In the stable range, the arithmetic Whitehead tower of the
symmetric groups 
is governed by the homotopy groups of spheres
\cite{EpaGanter17}, while that of $GL_n(\bbZ)$ is governed by the
$K$-theory of the integers, see Section \ref{sec:GL_n} below.
The inclusions
\[
  \Sigma_n\subset O_n(\bbZ) \subset GL_n(\bbZ).
\]
suggest to think of the Whitehead tower of the hyper-octahedral
groups as interpolating between these two.%
\begin{Thm}\label{thm:awt}
  For large enough $n$, the arithmetic Whitehead tower of the
  hyper-octahedral groups starts with the terms
  \[
    O_n(\bbZ) \longleftarrow A_n\ltimes (\bbF_n^{ev})\xleftarrow{\,\,(\bbZ/2)^2\,\,}
    \widetilde A_n\ltimes F_n^{ev} \longleftarrow \dots
  \]  
\end{Thm}
\begin{proof}
For $n\geq 2$, the abelisation of $O_n(\bbZ)$ is a Klein 4-group. This
is discussed, for instance, in \cite{Stembridge92}.
The index 2 subgroups of the hyper-octahedral group consist of
\begin{enumerate}
  \item the group $SO(\bbZ)$ of elements with even determinant,
  \item the group $A_n\ltimes\bbF_2^n$ of signed even permutations and
  \item the group $D_n=\Sigma_n\ltimes (\bbF_2^n)_{ev}$ of evenly
    signed permutations.
\end{enumerate}
The latter is a Weyl group of type $D$.  
The intersection of these three groups is the group
\[
  H = A_n\ltimes (\bbF_2^n)^{ev}
\]
of evenly signed even permutations. This is the commutator subgroup of
$O_n(\bbZ)$. We shall see below that it is perfect and hence possesses
a universal central extension. 
Write $E$ for the elementary 2-group $(\bbF_2^n)^{ev}$. Recall that
the universal coefficient theorem identifies the Pontryagin dual of
the integral homology with cohomology with circle coefficients
\[
  \widehat{H_i(G)} \cong H^i(G;U(1)),
\]
and we will switch between these two as convenient for calculating the
Schur multiplier of $H$.
Consider
the Lyndon-Hochschild-Serre spectral sequence
for the $U(1)$-cohomology of the extension
\[
  0\to E\longrightarrow H\to A_n\longrightarrow 0.
\]
We have 
\[
  E_2^{0,1} = (E^*)^{A_n} 
            =  0
\]
for $n\geq 3$, 
while
\[
  E_2^{0,2}((\Lambda^2E)_{A_n})^* = (\Lambda^2E^*)^{A_n}
\]
has two elements. The non-trivial element is the bilinear form
$$\beta = \sum_{1\leq i<j\leq n} x_i\wedge x_j,$$
where $(x_i)_i$ is the dual basis to $(e_i)_{i}$. Here we have used
the results in
\cite{Romagny05} to identify the invariants under the alternating group.
Note that $\beta$ is a cocycle for $F_n^{ev}$. This is checked on the
Clifford words of length 2, which generate $F_n^{ev}$.
On the other axis, we have $E^2_{1,0}=0$ for
$n\geq 5$, since the alternating groups become perfect in this
range. Further, for $n\geq 8$, the Schur multiplier of
the alternating group, $E_2^{2,0}$, has order 2. The non-trivial
element classifies the spin extension $\widetilde A_n$.
We now turn our attention to the entry $E_2^{1,1}$, which we claim to
be zero. Then the proof will follow from Corollary \ref{cor4.3} below.
\end{proof}
\begin{Lem}\label{lem:crossed_homomorphism}
  Let $f$ be a crossed homomorphism from $A_n$ to $E^*$ with respect to the
  action permuting the basis elements of $E$. Then $f$ is principal.
\end{Lem}
\begin{proof}
The alternating
group is generated by 3-cycles. 
So, if a crossed morphism is zero on all 3-cycles, it is zero.
Without loss of generality, we therefore assume that $f(123)\neq
0$.
We will continue to identify $E^*$ with the degree one
polynomials in the variables $x_i$ with $\bbF_2$-coefficients, modulo
the elementary symmetric polynomial $x_1+\dots+x_n$.
In this notation, $f(123)$ and $f(132)$ each have two summands, with
indices from the set $\{1,2,3\}$. Of course, there is a degree of
freedom here, we could add $x_1+\dots+x_n$ to obtain the other
representative; we choose the former representation.
To be precise, if $f(123)=x_i+x_{i+1}$ (indices modulo
3), then $f(132) = {}^{(132)}(x_i+x_{i+1}) = x_{i+1}+x_{i+2}$.
We will define a polynomial $p=\sum_{i=1}^n \alpha_ix_i$ such that $f$
agrees with the principal crossed homomorphism associated to $p$.
We set $\alpha_i=0=\alpha_{i+2}$ and $\alpha_{i+1}=1$.
If $n=3$, we are done. Otherwise, for any $k>3$, we
consider $f((12)(3k))$. This either contains both $x_3$ and $x_k$ as
summands, or neither. In the first case, we set $\alpha_k=\alpha_3+1$,
in the second case, we set $\alpha_k=\alpha_3$. This is
a generating set on which $f$ now agrees with the principal crossed
homomorphism associated to $\sum_{i=1}^n \alpha_ix_i$. 
\end{proof}
\begin{Cor}\label{cor4.3}
  For $n\geq 5$, the group $H$ is perfect.
  For $n\geq 8$, the Schur extension of $H$ is given by the
  semi-direct product 
  \[
    \widetilde A_n\ltimes F_n^{ev}.   
  \]
\end{Cor}
\begin{proof}
  The spectral sequence calculation above implies that $H$ is perfect
  and that its Schur multiplier has order $4$. Comparing to the Schur
  multiplier of the hyper-octahedral group $S_n\ltimes \bbF_2^n$, which
  was identified by Stembridge in \cite{Stembridge92} as an
  elementary abelian 2-group of order 8, we conclude that $H_2(H)$ is
  in fact a Klein 4-group. To be precise, this follows from 
  the spectral sequence
  for the extension
  \[
    1\longrightarrow H \longrightarrow S_n\ltimes
    \bbF_2^n\longrightarrow \{\pm1\}^2\longrightarrow 1,
  \]
  taking into account that $H_2(\{\pm1\}^2)$ has order 2.
  Write $\widetilde H$ for the Schur extension of $H$.
  By the universal property of Schur extension, we obtain a map of
  central extensions
  \[
    \begin{tikzcd}
      0\ar[r]& H_2(H)\ar[r] \ar[d] & \widetilde H\ar[r]\ar[d] &
        H\ar[d,equals]\ar[r]&0\\
      0\ar[r]& \{\pm1\}^2 \ar[r] & \widetilde A_n\ltimes F_n^{ev}\ar[r] &
        A_n\ltimes (\bbF_2^n)^{ev}\ar[r]&0\\
    \end{tikzcd}
  \]
  The arrow between the central subgroups must be an isomorphism,
  since the bottom extension cannot be reduced to an extension by $\{\pm1\}$.
  It follows that the middle arrow is an isomorphism, as well.
\end{proof}

As the
Schur extension of the perfect group $A_n\ltimes E$, the group
$\widetilde A_n\ltimes F_n^{ev}$ is super perfect and hence possesses a
categorical Schur extension, giving
the next step in the arithmetic Whitehead tower of the
hyper-octahedral groups (for $n\geq 8$). 
In the stable range,
we have maps
$$\pi_3(\bbS^0) \longrightarrow \widetilde A_n\ltimes
F_n^{ev}\longrightarrow \pi_3(K\bbZ).$$
It is therefore natural to conjecture that the categorical Schur
multiplier of $\widetilde A_n\ltimes F_n^{ev}$ also stabilizes for
large enough $n$ and that its order is a multiple of 24.

\subsection{Arithmetic Whitehead tower for $GL_n(\bbZ)$}
\label{sec:GL_n}
Mason and Stothers
  \cite[Theorem 5.1]{MR338209} have shown that for $n\geq3$ the
  subgroup $E_n(\bbZ)$ of   $SL_n(\bbZ)$ generated by the elementary
  matrices is perfect and coincides with the commutator subgroup of
  $GL_n(\bbZ )$.
We have
\[
  SO_n(\bbZ) \subset E_n(\bbZ).
\]
  The elementary matrices satisfy the following
  relations 
\begin{eqnarray}\label{1}
e_{ij}(a)e_{ij}(b)&=&e_{ij}(a+b)\\\label{2}
e_{ij}(a)e_{kl}(b)&=&e_{kl}(b)e_{ij}(a),\; j\not=k \mbox{ and }i\not=l \\ \label{3}
e_{ij}(a)e_{jk}(b)e_{ij}(a)^{-1}e_{jk}(b)^{-1}&=&e_{ik}(ab);\ i,j,k \mbox{ distinct } \\ \label{4}
e_{ij}(a)e_{ki}(b)e_{ij}^{-1}(a)e_{ki}(b)^{-1} &=& e_{kl}(-ba); \;  i,j,k \mbox{ distinct }
\end{eqnarray}
The universal central extension of $E_n(\bbZ)$ is known as Steinberg
group $St_n(\bbZ )$. It is defined by symbols $x_{i j}(a)$ with
$i\not= j, 1\leq i,j,\leq n, a\in \bbZ $ subject to the relations
(\ref{1})(\ref{2})(\ref{3})(\ref{4}). The kernel of the map 
$$St_n(\bbZ )\lra E_n(\bbZ ); \; x_{ij}(a) \mapsto e_{ij}(a)$$
is the algebraic K homology group $K_2(\bbZ )$. In particular, the
first two homology groups of $St_n(\bbZ )$ vanish and the Schur
multiplier is  
\[
 H_2(E_n(\bbZ ))\cong K_2(\bbZ ) \cong \bbZ/2.
\]
Using standard techniques from algebraic $K$-theory
\cite[5.2.7]{MR1282290} it is not hard to see that 
$$ H_3(St_n(\bbZ))\cong K_3(\bbZ) \cong \bbZ/48.$$
So, we have the tower
\[
    GL_n(\bbZ) \longleftarrow E_n(\bbZ) \xleftarrow{\,\,(\bbZ/2)^2\,\,}
    St_n(\bbZ) \xleftarrow{\,\,\mathbb B \bbZ/48 \,\,}  \mathcal St(\bbZ).
\longleftarrow \ldots
\]
The group $K_3(\bbZ)$ plays a key role in the computation of
the fourth cohomology of the Conway group
\[H^4(Co_1;\bbZ)\cong\bbZ/24 \]
by Johnson-Freyd and Treumann \cite{Johnson-FreydTreumann20}.
\subsection{Integers in cyclotomic fields and string extensions.}
We have seen that good codes lead to interesting representations of
the integral spin groups. We like to see similar phenomena for the
integral string groups. For this purpose,  we consider the complex
Whitehead tower and restrict it to integers in cyclotomic fields. \par 
Let $U_n(\bbC)$ be the group of unitary matrices, that is, complex
isometries of $\bbC^n$. 
Since the fourth cohomology of $BSU_n(\bbC)$
is generated by the second Chern class we obtain a $PU(\mbox{H})$-extension
$\widetilde{SU}_n(\bbC)$ in the same way as in the real case.
t
Let $\bbO$ be the ring of integers in the number field $K=\bbQ(\zeta)$
considered in Section \ref{Lattices}. The integral unitary group 
$$U_n(\bbO)=U_n(\bbC)\cap Gl_n(\bbO)$$ coincides with the semi direct
product of diagonal with permutation matrices  $$U_n(\bbO)=\{\pm
\zeta^i| \, i=0, \ldots , p-1\}^n\rtimes \Sigma_n.$$  A code of length
$n$ can be considered as subset of $U_n(\bbO)$ by identifying  
$\bbF_p$ with the multiplicative group of roots $ \zeta^i$.

Consider the pullback diagram
$$\begin{tikzcd} (\widetilde{\bbF_p} )^n_{det =1}
  \ar[r,hook]\ar[d,swap,"PU(\Hi )"] &\widetilde{SU}_n(\bbO)\ar[d,"PU(\Hi )"]
  \ar[r]&\widetilde{SU}_n(\bbC)\ar[d,"PU(\Hi )"] 
  \\[+3ex]
  (\bbF_p )^n_{det =1}  \ar[r,hook] &SU_n(\bbO)\ar[r]& SU_n(\bbC).
\end{tikzcd}
$$
There are several ways to deal with representations of $PU(\mbox{H})$-extensions of topological groups $G$. One way is to take based loops $\Omega$ and observe that $\Omega PU(\mbox{H})\cong S^1 $. Hence,
$\Omega \widetilde{G}$ is equivalent to an extension of  $\Omega {G}$ by a compact Lie group.  Then one looks for representations which mimic the structure of the loop space of a representation of $\widetilde{G}$. 
\par
This approach does not work in the arithmetic setting since the base spaces are discrete. In our case, however, we may employ the short exact sequence 
$ \bbP \rightarrow \bbO \stackrel{\rho}{\rightarrow} \bbF_p$ of Section \ref{Lattices} 
to obtain the fibre sequence
$$ \bbF_p^n\longrightarrow B \bbP^n \longrightarrow B\bbO^n .$$
Explicitly,  the topological group  $B\bbP$ is the  $(p-1)$-dimensional torus $$T_\bbP =(\bbO \otimes \bbR) /(\bbP \otimes \bbZ) $$ 
and $\bbF_p$ sits in $T_\bbP$ as $\bbO \otimes \{1\} / \bbP \otimes \{ 1\}$. The inclusion of $\bbF_p$ in $U_1(\bbC) \cong \bbR/\bbZ $ factorizes through  $T_\bbP$ by the map 
\begin{eqnarray}
 (a_0+a_1\zeta +\cdots +a_{p-2}\zeta^{p-2}) \mapsto  (a_0+a_1 +\cdots +a_{p-2})/p.
\end{eqnarray}
with $a_i \in \bbR$.  If one chooses the basis 
\begin{eqnarray}\label{basis}
 &1-\zeta , \zeta-\zeta^2, \ldots , \zeta^{p-2}-\zeta^{p-1} &
 \end{eqnarray} of $\bbP$ then the  induced map
$$  (\bbR /\bbZ)^{p-1} \cong  \bbR /\bbZ \otimes \bbP =T_\bbP\stackrel{\rho}{\longrightarrow} \bbR/\bbZ$$
is the projection to the last factor. 
\begin{Lem}
The extension of $({\bbF_p} )^n_{det =1}$ defined by the quadratic form of the lattice $\bbP^n$ coincides with $(\widetilde{\bbF_p} )^n_{det =1}$ which was defined by the second Chern class.  \end{Lem}
\begin{proof}
The quadratic form of the lattice $\bbP$ with respect to the basis (\ref{basis}) has the form
$$ q(x)=\sum_{i=0}^{p-1}x_i^2 +x_ix_{i+1}.$$
As a 4-dimensional cohomology class it restricts to a generator $z^2$ in $H^4B\bbF_p$ by what we said before. The computation
$$\sum_i z_i^2 = \sigma_1^2 -2\sigma_2$$
in the elementary symmetric polynomials $\sigma_i$  implies that the quadratic form of orthogonal sum $\bbP^n$ corresponds to $c_2=\sigma_2$ up to the unit -2 mod $p$  once $c_1$ vanishes.
\end{proof}

In the pull back diagram above, one may hence replace the group
$\bbF_p^n$ with the torus $T_\bbP^n$ and look for loop representation
of the extension 
$\widetilde{\Omega {T_\bbP^n}}$ with the center $S^1 $ acting non trivially. 
This will be done in the next section in the infinitesimal setting. The role of codes in $\bbF_p^n$ will then be disclosed.
\section{Lattice vertex operator algebras and their representations}\label{LVOA}

The mathematical notion of a vertex algebra goes back to the work of Borcherds in 1986 \cite{MR843307} who axiomatized the relations amongst lattice vertex operators.
The definition appears rather technical but  will become more explicit when we look at the example 
$\widetilde{\Omega {T_L}}$ where $L$ is a lattice. 
\subsection{Vertex operator algebras and modules}
A {\em vertex  algebra} consists of the following data:
\begin{enumerate}
\item a ${\bbZ}_+$-graded vector space $V$
\item a vacuum vector $\bbone \in V_0$, $\bbone \not= 0$
\item a linear map $D:V \ra V$ of degree one, $D\bbone =0$
\item a linear map $Y(\cdot, z):V\ra \mbox{End}(V) [\![ z, z^{-1}]\!]$, 
$Y(v,z)=\sum v_n z^{-n-1}$
\end{enumerate}
The endomorphisms $v_n$ are called modes of $v$. These data should satisfy the following properties for all $v,w\in V$ (see \cite{MR1886764}\cite{MR2648364} for details) : 
\begin{enumerate}
\item there is an integer $N=N(w)$ with $v_n(w)=0$ for all $n >N$
\item $ Y(v,z)\in V [\![ z, z^{-1}]\!][z^{-1}]$ and there is a $k\geq 0$ with $$(z_1-z_2)^k[Y(v_1,z_1),Y(v_2,z_2)]=0.$$ This property is called locality and one writes $$ Y(v,z)\sim Y(w,z)$$
if a power of $(z_1-z_2)$ annihilates the bracket.
\item $Y(v,z)\bbone = v+O(z)$
\item $[D,Y(v,z)]=\partial Y(u,z)$ where $\partial$  is the formal derivative with respect to $z$
\end{enumerate}
Note that the operator $D$ is entirely determined by the state-field correspondence $Y$:  property (iv) implies $Dv=v_{-2}\bbone$. 
\par
A {\em vertex operator algebra} is a vertex algebra  which is  equipped with a distinguished state $\omega \in V$. Its modes $L_n=\omega_{n+1}$ of $\omega$ 
should satisfy
\begin{enumerate}
\item $V_n=\{ v \in V| \, L_0v=nv\} $
\item $Y(L_{-1} v,z)=\partial Y(v,z)$
\item $[L_m,L_n]=(m-n)L_{m+n}+\frac{m^3-m}{12}\delta_{m,-n}c\,  id_V$
for some number $c$.
\end{enumerate}
In addition, one requires  $\dim V_n<\infty$, $V_n=0$ for $n<\! \! \!<0$. The state  $\omega$ also carries the name {\em conformal vector}. (It is responsible for the action of the Virasoro algebra, that is, the reparametrization action of loops.)

A {\em module over a vertex operator algebra} is a vector space $M$ together with a linear map
$$ Y_M: V \ra \mbox{End}(M) [\![ z, z^{-1}]\!]; \, v \ra Y_M(v,z)=\sum_n v_n^M z^{-n-1}$$ with the properties:
\begin{enumerate}
\item $Y(\bbone ,z)=id_M$
\item $Y_M(v,z)\sim Y_M(w,z)$
\item for $k >\! \! \!>0$  associativity holds, that is, 
$$ (z_1+z_2)^kY_M(v,z_1+z_2)Y_M(w,z_2)=(z_1+z_2)^kY_M(Y_M(v,z_1)w,z_2).$$
\end{enumerate}
Usually one requires another grading $M=\bigoplus_{\lambda \in \bbC}M_\lambda$ such that
$\dim M_\lambda< \infty$, $M_{\lambda +n}=0$ for $n<\! \! \!<0$ and $L_0m =\lambda m $ for $m\in M_\lambda$. It turns out that irreducible modules over tame vertex operator algebras have this property. Here, a $V$-module is irreducible if it contains  no proper, nonzero submodule (invariant under all modes $v_n^M$). 
\par
Note that if $M$ is irreducible and if $M_{\lambda+n}\not=0$ for some $n\in \bbZ$ then 
$$ M= \bigoplus_{n\in \bbZ}M_{\lambda +n}$$
because the right module is invariant under all modes. Set $h=\lambda+n$ with the smallest $n$ such that $ M_{\lambda +n}\not=0$. This number turns out to be rational and it is called the conformal weight  of $M$. \par
The {\em partition function} of a representation $M$ is the function on the upper half plane given by the formula 
$$ Z (z ) = tr_M q^{L_0^M-c/24}=q^{h-c/24}\sum_{n\geq0} \dim(M_{h+n})q^n$$
where $q=e^{2\pi i z}$. 
\subsection{Lattice algebras}
We now specify to the vertex  algebra associated to a $d$-dimensional lattice $L$ with a non degenerate symmetric bilinear form $q$. We will construct  an infinitesimal  model for $\widetilde{\Omega {T_L}}$ where $T_L=L\otimes S^1$. \par
Set ${\textgoth h}=\bbC^d$ and think of ${\textgoth h}$ as the abelian complexified Lie algebra of $T_L$. Define the Heisenberg algebra
$$\hat{\textgoth h}= {\textgoth h} \otimes \bbC [t,t^{-1}] \oplus \bbC K$$
with central element $K$ and  brackets
$$[v\otimes t^m,w\otimes t^n]=m \left< v,w \right> \delta_{m,-n}K.$$
The first summand in $\hat{\textgoth h}$ corresponds to loops in $T_L$ and the second one gives a central extension by $S^1$. This already is a Lie algebra model for  $\widetilde{\Omega {T_L}}$. However, the loop space of a group and its extension has more structure than just being a group. For example, one can reparameterize loops. This action should be incorporated into the  model. \par
The Heisenberg algebra comes with a triangular decomposition given by 
$$ \hat{\textgoth h}^\pm=\oplus_{\pm n>0}{\textgoth h}\otimes t^n; \; \hat{\textgoth h}^0={\textgoth h}\oplus \bbC K.$$
Let $M$ be a  ${\textgoth h}$-module $X$.  Let $l$ be a scalar. Extend the action of ${\textgoth h}$ to $\hat{\textgoth h}^+\oplus \hat{\textgoth h}^0$ by letting $\hat{\textgoth h}^+$ annihilate $M$ and by letting  $K$ act by multiplication with the level $l$. Let $V_{\textgoth h}(l,M)$ be the induced $\hat{\textgoth h}$-module. Explicitly, 
\begin{eqnarray}\label{Borcherds}
 V_{\textgoth h}(l,M) = {\mathcal U}(\hat {\textgoth h})\otimes_{{\mathcal U}(\hat {\textgoth h}^0\oplus \hat {\textgoth h}^+)}M \cong {\mathcal U}(\hat {\textgoth h}^-)\otimes M \cong S(\hat {\textgoth h}^-)\otimes M
 \end{eqnarray}
where ${\mathcal U}$ denotes the universal enveloping algebra and $S$ is the symmetric algebra. Set 
$$ Y(v,z)=\sum_nv_{(n)}z^{-n-1}$$
and $v_{(n)}$ is the action  of $v\otimes t^n$ on $V_{\textgoth h}(l,M)$. This state-field correspondence $Y$ makes  $$ M_0^d\stackrel{def}{=} V_{\textgoth h}(1,\bbC \bbone)$$ into a vertex algebra. The conformal vector is
$$ \omega = \frac{1}{2} \sum_{i=1}^d (v_i)_{ (-1)}(v^i)_{(-1)}  \bbone$$
and central charge is $d$. Here, $\{ v_i\}$ is a basis of $\bbC^d$ and $\{ v^i\}$ is the dual basis with respect to $Q$.
Moreover, $V_{\textgoth h}(1,M)$ is a module over $M_0^d$  (see \cite{MR2648364} for details).\par
Now assume that $
L$ is an even lattice in ${\textgoth h}=\bbR^d$ with respect to a  positive definite real quadratic form $q_{\textgoth h}$ and let ${\textgoth h}$ and $q$ be the complexifications. 
The irreducible representations $M_\alpha$ of ${\textgoth h}$ are one-dimensional and indexed  by elements of the dual space $\alpha \in {\textgoth h}^*$. Set
$$V_L=\bigoplus_{\alpha \in L} V_{\textgoth h}(1,M_{\alpha^\vee})\cong S(\hat{\textgoth h}^-)\otimes \bbC[L]$$
where $\alpha^\vee$ denotes the dual of $\alpha$.
The construction makes $V_L$ into a module over the Heisenberg vertex algebra $M_0^d$. 
\par
With some effort, the vector space $V_L$ can itself be given the structure of a vertex operator algebra:
for $e^\alpha \in M_{\alpha^\vee}$ set
$$ Y(e^\alpha,z)=\exp(\sum_{n>0}\frac{\alpha_{(-n)}}{n}z^n) \exp(\sum_{n<0}\frac{\alpha_{(-n)}}{n}z^n) e^\alpha z^\alpha.$$
Here,  $\alpha_n$ are the modes of $Y(\alpha,z)$ in $M_0^d$, $z^\alpha : v \otimes e^\beta \mapsto z^{q(\alpha,\beta)}v\otimes e^\beta$ for $v \in \hat{ {\textgoth h}^-}$ and $e^\alpha:v \otimes e^\beta \mapsto \epsilon(\alpha,\beta)v \otimes e^{\alpha + \beta}$ for a certain bilinear 2-cocycle $\epsilon:L\otimes L \ra \{ \pm 1\} $.
This imposes the structure of a vertex algebra on $V_L$. See the textbook \cite{MR1651389} or \cite{MR996026} for details.

\begin{Pro}(see \cite[Equation (75)]{MR2648364})\label{Dongtheta}
The partition function of $V_L$ is 
$$  Z_{V_L}(z)= \eta(q)^{-d} \sum_{\alpha\in L}q^{\left< \alpha, \alpha \right>/2 }=
\eta(q)^{-d} \vartheta_L(q) $$
where $\eta$ is the Dedekind $\eta$-function
$$\eta(q)=q^{1/24}\prod_{n\geq 1}(1-q^n).$$
\end{Pro}

The irreducible modules of $V_L$ were computed by Dong. They are indexed by elements of the set $C_f=L^\vee/L$ where $L^\vee$ denotes the dual lattice
$$ L^\vee = \{ \beta \in \bbR^d |\, Q(\alpha, \beta)\in \bbZ \mbox{ for all }\alpha \in L\}.$$
For $\lambda \in C_f$ the module  takes the form
$$V_{L+ \lambda } =\bigoplus_{\alpha \in L}V_{\textgoth h}(1, M_{(\alpha+\lambda)^\vee})\cong
S(\hat{{\textgoth h}}^-)\otimes \bbC [L+\lambda ]$$
and its partition function is
$$ Z_{V_{L+\lambda }}(z)= \eta(q)^{-d} \sum_{\alpha\in L}q^{\left< \alpha + \lambda,  \alpha + \lambda\right>/2 }.
$$
\begin{thm}[\cite{MR1245855}]\label{Dong}
The $V_L$-modules  $\{V_{L+\lambda}\}$ for  ${\lambda \in C} $ are all distinct and they provide a complete list for the isomorphism classes of  irreducible $V_L$-modules. 
\end{thm}

\subsection{The  Eisenstein vertex algebra and its representations}
Let $\bbO$ be the ring of integers of the cyclotomic field $\bbQ[\zeta]$ as considered before.  Let $C$ be a linear code in $\bbF_p^n$ with $C\subset C^\perp$.  
In section \ref{Lattices} we defined the even lattice $\Gamma_C\subset \bbO^n$ with bilinear form
$$ \langle x,y\rangle =\sum_{i=1}^n \mbox{Tr}\left( \frac{x_i\bar{y}_i}{p}\right).$$
Its dual is the lattice $\Gamma_{ C^\perp}$. In particular, for $L=\Gamma_{0}=\bbP^n$ this means
$$(\bbP^n)^\vee=\bbO^n. $$
In the following, the associated vertex operator algebra $V_n=V_{\bbP^n}$ will be called {\em  Eisenstein  vertex algebra}. 
Theorem \ref{Dong} tells us that the isomorphism classes of irreducible representations of $V_n$ are indexed by words $w$ of the full code $C_f=\bbO^n/\bbP^n\cong \bbF_p^n$. 
\par
Let $\mR ep(V_n)$ be the free group generated by these isomorphism classes. 
Lattice algebras associated to positive-definite even lattices are rational, that is, any  module is a direct sum of irreducible modules (see \cite{MR2097833}). These sums are necessary finite because of the restrictions on the dimensions. We conclude that $\mR ep(V_n)$ coincides with the Grothendiek group of the category of $V_d$-modules. 
Set $$\mR ep(V)=\sum_{n\geq 0}\mR ep(V_n)$$ with $\mR ep(V_0)= \bbZ$. It has a ring structure since
$$\mR ep (V_{{n+n'}})\cong \mR ep (V_{n}) \otimes \mR ep (V_{{n'}}).$$
The automorphism group 
$\mbox{Aut}(C_f)=(\bbF_p^\times)^n \rtimes \Sigma_n$
acts on $\mR ep (V_{n})$ in the obvious way. It contains the subgroup  $$\mbox{Aut}_\bbR(C_f)= \{\pm 1\}^n \rtimes \Sigma_n $$   of {\em real} automorphisms. \par
Before formulating the main result we need to introduce some notation. For a ring $R$  let $\mR ep _R(V)$ be the product $\mR ep (V)\otimes R$. It is the free $R$-module generated by the irreducibles. Let  $mf(p)$ be the ring of modular forms for the congruence group 
$\Gamma(p)$
with $q$-expansion at $\infty$ in $\bbO[\frac{1}{p}][\![ q^{1/p}]\!]$. Similarly, write $hmf^{\Sigma}({\textgoth p})$ for the  $\bbO[\frac{1}{p}]$-algebra of symmetric Hilbert modular forms for $\Gamma({\textgoth p})$ generated by $\theta_j$ for  $j=0, \ldots, ,r$. 

\begin{Exa}\label{35}
For $p=3$, the ring $mf(3)_\bbC$ is freely generated by $\theta_0,\theta_1$  (see  f.i.\cite[Theorem 5.4]{MR2977354}). They have $q$-expansions at $\infty$
\begin{eqnarray*}
\theta_0&=&\sum_{(x,y)\in \bbZ^2}q^{x^2-xy+y^2} = 1+6(q+q^2+q^4+2q^7\ldots)\\
\theta_1& =&q^{\frac{1}{3}}\sum_{(x,y)\in \bbZ^2}q^{x^2-xy+y^2+x-y}=
3q^{\frac{1}{3}}(1+q+2q^2+2q^4+\ldots ).
\end{eqnarray*}
The first two coefficients of $\theta_0$ and $\theta_1$ imply that $mf(3)$ itself is freely generated by $\theta_0,\theta_1$and hence 
\begin{eqnarray}\label{hmf3}
 hmf^{\Sigma}({\textgoth 3}) =mf(3)= \bbO[{\textstyle \frac{1}{3}}][\theta_0, \theta_1].
 \end{eqnarray}
For $p=5$,
Hirzebruch has proved in \cite{MR0480355} that 
the ring of complex symmetric Hilbert modular forms is a polynomial ring in $\theta_0, \theta_1, \theta_2$. This implies
\begin{eqnarray}\label{hmf5}
 hmf^{\Sigma}({\textgoth 5})\cong \bbO[{\textstyle \frac{1}{5}}][\theta_0, \theta_1, \theta_2].\end{eqnarray}
\end{Exa}

 \begin{thm} \label{Main}
 
 \begin{enumerate}
\item  
The partition function induces a well defined ring map 
\begin{eqnarray*} Z: \mR ep_{\bbO[\frac{1}{p}]} (V)/\mbox{Aut}_\bbR(C_f)&\stackrel{}{\lra}& mf(p)\\
\left[ M \right] &\mapsto& \eta^{n(p-1)}Z_M
\end{eqnarray*} 
\item The map $Z$ is the composite of a map of graded rings 
$$ \tilde{Z}: \mR ep_{\bbO[\frac{1}{p}]} (V)/\mbox{Aut}_\bbR(C_f) \longrightarrow hmf^{\Sigma}({\textgoth p})$$ 
with the map induced by the diagonal.
\item 
The map $\tilde{Z}$ is an isomorphism  for $p=3$ and  $p=5$. 
\end{enumerate}
\end{thm}
\begin{proof} 
Irreducible representations of $V_n$ are indexed by words $w\in \bbF_p^n$ and have the partition function $\eta^{-n(p-1)}   \vartheta_{\{ w\}}$ by Proposition \ref{Dongtheta}. This expression coincides with the diagonal of 
$$  \theta_{\{ w\}}= \theta^{l_0(c)}_0\theta_1^{l_1(c)}\cdots \theta_{r}^{l_r(c)}(z)$$
by Theorem \ref{GeerHirz} and  hence is invariant under the action of $\mbox{Aut}_\bbR(C_f)$. Since each $\theta_j $ lies in $hmf^{\Sigma}({\textgoth p})$ we obtain a well defined map $\tilde{Z}$ and the first two assertions are shown. \par
For the last claim, observe that the target is a polynomial ring in the $\theta_j$s by Equations  (\ref{hmf3})(\ref{hmf5}) and the map sends a representation indexed by symmetrized word to the corresponding monomial.
\end{proof}
\begin{Cor}
For a code $C\subset \bbF_p^n$ define the $V_n$-module
$ M_C=\sum_{w\in C}V_{\bbP^n +w} $. Suppose $C$ is self dual.
\begin{enumerate}
\item
The partition function $Z$ of $M_C$ is a $SL_2(\bbZ)$-invariant modular form.
\item
The refined partition function $\tilde{Z}$ is a $SL_2({\textgoth O})$-invariant Hilbert modular form for $p=5$.
\end{enumerate}
\end{Cor}
\begin{proof}
The first statement follows from Theorem \ref{Main} and Corollary \ref{Sl2}. The second statement for $p=5$ follows from Theorem \ref{Main} and a result of Gleason-Pierce and Sloane (see \cite{MR398664},\cite[Corollary 5.4,5.5]{MR2977354}).
\end{proof}
\begin{Rem}
For $p\geq 5$ the  refined partition function $\tilde{Z}$ should have a more intrinsic description in terms of traces of Eisenstein vertex operators. Also, it would be worth to know
the representations which allow the Doi-Naganuma lift (c. Remark \ref{DN}). 
\end{Rem}
\begin{Rem}
There is an action of $SL_2(\bbZ)$ on the ring of integral $\Gamma(3)$-modular forms: if $\theta$ is a modular form of weight $n$ and level 3 then
$$\theta\left(\frac{az+b}{cz+d}\right)(cz+d)^{-n}$$ 
is again a modular form of weight $n$  and level 3. Since the subgroup $\Gamma(3)$ acts trivially it is an action of $$G_1=SL_2(\bbF_3)\cong SL_2(\bbZ)/\Gamma(3).$$ Explicitly, the generators $S$ and $T$ act on $\theta_0$  by
\begin{eqnarray*}
\theta_0\left(\frac{-1}{z}\right)&=&\frac{z (-1-2\zeta)}{3}(\theta_0(z)+2\theta_1(z))\\
\theta_0(z+1)&=&\theta_0(z)
\end{eqnarray*}
and on $\theta_1$ by
\begin{eqnarray*}\label{action S}
\theta_1\left(\frac{-1}{z}\right)&=&\frac{z (-1-2\zeta)}{3}(\theta_0(z)-\theta_1(z))\\
\theta_1(z+1)&=&\zeta \theta_1(z).
\end{eqnarray*}
The fix point ring  
\begin{eqnarray}\label{fix}
mf(3)^{G_1}&=& mf_{\bbO[\frac{1}{3}]}
\end{eqnarray}
coincides with the ring of integral  modular forms for the full group
$SL_2(\bbZ)$. \end{Rem}

 \section{Topological modular forms with level structure}\label{section tmf}
 Atiyah-Bott-Shapiro have discovered the role of Clifford algebra
 representations in $K$-theory  \cite{MR0167985}. We have seen in
 Section \ref{AWT} that these correspond to representations of the
 finite group $F_n$ in the arithmetic Whitehead tower and are best
 understood using coding theory. In Section \ref{LVOA}, we investigated
 the representations  of the string extension of the corresponding
 group in the complex arithmetic Whitehead tower. Theorem \ref{Main}
 gives a correspondence of its representation ring with modular forms
 for the congruence subgroup $\Gamma(N)$. In this section, we like to
 interpret this result as a higher version of the
 Atihay-Bott-Shapiro construction. The higher version of $K$-theory will be  $TMF(N)$,
 the cohomology theory of topological modular forms with full level
 $N$ structure. \par 
\subsection{The spectrum $TMF(N)$} Elliptic cohomology has its origin in the study of elliptic genera. F.\ Hirzebruch \cite{MR981372} generalized the earlier notion of S.\ Ochanine from level 2 to higher levels and constructed ring maps
$$ \chi_*:  MU_* \longrightarrow mf(N)_* $$
where $MU_* $ is the bordism ring of stably almost complex manifolds. These elliptic  genera 
come along with rigidity results for the equivariant index of the Dirac operator on  manifolds with circle actions. E.\ Witten related them to index theory on  loop spaces of manifolds with string structures. The Witten genus is a ring map  
$$ W: MString_* \longrightarrow mf_* $$
where $MString$ denotes the cobordism theory of manifolds with string structures. 
 \par
In the late 80s, Landweber-Ravenel-Stong \cite{MR1320998} introduced a
cohomology theory, nowadays denoted by  $TMF_1(2)$, for the congruence
group $\Gamma_1(2)$. It was generalized to other levels by A.\ Baker
\cite{MR1271552}, J.L.\ Brylinski \cite{MR1071369} (with 2 inverted)
and by the second author with the help of a result by J.\ Franke \cite{MR1235295}.
The method given in
\cite{MR1660325} shows that the Hirzbruch genera define periodic
homology and cohomology  theories given by 
\begin{eqnarray}\label{CF}
 TMF(N)_*(X)&=& MU_*(X) \otimes_{\chi_*}mf(N)_*[\Delta^{-1}].
 \end{eqnarray}
The corresponding ring spectra come with 
 ring maps $$\chi : MU \rightarrow TMF(N)$$
but these complex orientations are not compatible amongst each other along the maps
$TMF(N)\ra TMF(NM)$.
  A family version of the Witten genus $$W: MString \longrightarrow tmf$$ was achieved by M.\ Hopkins et al.\cite{MR1989190}\cite{MR3223024}. Its target is the spectrum of topological modular forms $tmf$ (or $TMF$ for the periodic version). Its coefficient ring  maps to the ring $mf_*$ of integral modular forms. This map is neither surjective nor injective but rationally it is an isomorphism.  \par
There are basically two different constructions of the spectrum $tmf$ as a highly structured ring spectrum. One construction is based on its localizations with respect to Morava $K$-theories (see \cite[Chapter 11]{MR3223024}). It uses the fact that its $K(2)$-localization is the Lubin-Tate theory $E_2$. This theory originates from deformations of the elliptic curve mentioned in Section \ref{sec:LT}.  Its $K(1)$-localization is related to the theory of $p$-adic modular forms (see \cite{MR2076927}). The other construction uses derived algebraic geometry, that is, algebraic geometry with ordinary rings replaced by  highly structured ring spectra. The reader is referenced to the survey article \cite{MR2597740}. For a construction of the connective theories $tmf(N)$ along these lines see \cite{MR3455154}.
\par
There is an equivalence of homotopy fixed point spectra
$$ TMF[1/p] \ra TMF(p)^{h {Gl}_2(\bbF_p)}$$
which should be regarded as a homotopy version of Equation (\ref{fix}). The spectral sequences
$$ H^s(Gl_2(\bbF_p), \pi_{2t}TMF(p))\Longrightarrow \pi_{2t-s}TMF[1/p]$$
where computed by Mahowald-Rezk in \cite{MR2508904} for $p=3$ and by Behrens-Omsky in \cite{MR3572338} for $p=5$.
\par
For $p=3$, the proposed generalization of the Atiyah-Bott-Shapiro
theorem is the following: 

\begin{Cor}\label{Main3}
\begin{enumerate}
\item
The partition function $Z$ gives an isomorphism
\begin{eqnarray*} \mR ep_{\bbO[\frac{1}{3}]} (V)_n/\mbox{Aut}_\bbR(C_f)&\stackrel{\cong}{\lra}& tmf(3)^{0}(S^{2n}).
\end{eqnarray*} 
\item If $C$ is a self dual ternary code of length $n$ then there is an element of $tmf^0(S^{2n})$ whose image in $tmf(3)^{0}(S^{2n})$ is $Z(M_C)$.
\end{enumerate}
\end{Cor}
\begin{proof}
The first claim only is a reformulation of Theorem \ref{Main} for the case $p=3$. The second claim follows from Equation (\ref{dual}) and \cite[Proposition 5.12]{MR1989190}.
\end{proof}
H.\ Tamanoi argued  in \cite{MR1716689} that the equivariant index of the formal Dirac operator on loop spaces is a virtual module over some vertex operator algebra. Of course, this is a consequence of Corollary \ref{Main3} for $p=3$, since the index coincides with the elliptic genus and hence takes values in the ring of $\Gamma(3)$-modular forms.  
This suggests  that  there is a geometric construction of the cohomology theory $tmf(3)$ along these lines.
\par
It turns out that bundles whose fibers are  representations of the Eisenstein vertex algebra lead to a cohomology theory $K_V$ by the usual construction: one takes the Grothendieck group of isomorphism classes and inverts the representation which is attached to the discriminant $\Delta$ to get a periodic theory. The corollary then says that $K_V$ has  the same  coefficients as $TMF(3)$. However, using results from  Dong-Liu-Ma-Zhou  \cite{MR2106231}, it is not hard to see that one only gets a form of $K$-theory
$$K_V(X)= K(X)\otimes mf(3)_*[\Delta^{-1}]$$
since the associated genus is  equivalent to the Todd genus. Bundles of representations hence do not lead to the correct elliptic cocycles. On the other hand, Equation (\ref{CF}) implies that a cohomology theory with the same coefficients and the correct Euler classes has to coincide with $TMF(3)$. 

\subsection{Sheaves of vertex operator algebras and Euler classes} 
 Bundles correspond to quasi-coherent sheaves. Gorbounov-Malikov-Schechtman-Vaintrob \cite{MR1748287} constructed  a more general sheaf $MSV(X)$  of vertex operator algebras   for each  complex manifold $X$.  It comes with a multiplication map
 $$ {\mathcal O}(X) \otimes MSV(X)\longrightarrow MSV(X) $$
by functions on $X$ which, however,   is not associative.   
 There is  a filtration of $MSV(X)$ compatible with the multiplication such that the associated graded is the coherent sheaf (cf.\cite{MR1757003})
 $$ ell_X= \Lambda_{y^{-1}}T^*X \otimes\bigotimes_{n=1}^\infty \Lambda_{yq^n}T^*X\otimes \bigotimes_{n=1}^\infty S_{q^n}T^*X\otimes
   \bigotimes_{n=1}^\infty \Lambda_{y^{-1}q^n}TX\otimes \bigotimes_{n=1}^\infty S_{q^n}TX.$$
 Take $y=-\zeta$ and consider $ell_X$ as an element of 
 $$K_{Tate(N)}(X) =K(X)_{\bbO[1/N]}((q^{1/N})).$$ 
 \begin{Pro}\label{Euler}
Up to a normalization factor,   $ell_X$ coincides with the Euler class of $TX$ in  $K_{Tate(N)}$ once $K_{Tate(N)}$ is equipped with the orientation coming from the 
  level $N$-elliptic character \cite{MR1660325} \cite{MR1781277}
 $$\lambda: TMF(N) \longrightarrow K_{Tate(N)}.$$
 \end{Pro}
 \begin{proof}
This can easily be verified with a Chern character computation (see \cite[p.117f]{MR1189136}).
\end{proof}
Proposition \ref{Euler} suggests that $MSV(X)$ should rather be considered as  an element of $TMF^{2n}(X)$, the elliptic Euler class,  which maps under $\lambda$ to $ell_X$. Up to a normalization factor, the elliptic character hence is the map which takes the associated graded object. 
\par
The normalization factor also plays a role in Corollary \ref{Main3} where the considered objects have the same spirit. In fact, the construction of $MSV(X)$ given in \cite{MR2294219} resembles Equation (\ref{CF}) where the module $M$ is replaced by a quasi-coherent sheaf. Lian-Linshaw even provided equivariant versions of $MSV(X)$ in the same article  for $X$ which are equipped with an action of a compact Lie group. Needless to say that such geometric interpretations in terms of vertex operator algebras would tremendously enhance the theory of topological modular forms and conformal field theories.
 \noindent 
\bibliographystyle{amsalpha}
\bibliography{bibCVOTMF}

\providecommand{\bysame}{\leavevmode\hbox to3em{\hrulefill}\thinspace}
\providecommand{\MR}{\relax\ifhmode\unskip\space\fi MR }
\providecommand{\MRhref}[2]{%
  \href{http://www.ams.org/mathscinet-getitem?mr=#1}{#2}
}
\providecommand{\href}[2]{#2}
\begin{thebibliography}{DLMZ04}

\bibitem[ABS64]{MR0167985}
M.~F. Atiyah, R.~Bott, and A.~Shapiro, \emph{Clifford modules}, Topology
  \textbf{3} (1964), no.~suppl, suppl. 1, 3--38. \MR{167985}

\bibitem[Bak94]{MR1271552}
Andrew Baker, \emph{Elliptic genera of level {$N$} and elliptic cohomology}, J.
  London Math. Soc. (2) \textbf{49} (1994), no.~3, 581--593. \MR{1271552}

\bibitem[BL00]{MR1757003}
Lev~A. Borisov and Anatoly Libgober, \emph{Elliptic genera of toric varieties
  and applications to mirror symmetry}, Invent. Math. \textbf{140} (2000),
  no.~2, 453--485. \MR{1757003}

\bibitem[BO16]{MR3572338}
Mark Behrens and Kyle Ormsby, \emph{On the homotopy of {$Q(3)$} and {$Q(5)$} at
  the prime 2}, Algebr. Geom. Topol. \textbf{16} (2016), no.~5, 2459--2534.
  \MR{3572338}

\bibitem[Bor86]{MR843307}
Richard~E. Borcherds, \emph{Vertex algebras, {K}ac-{M}oody algebras, and the
  {M}onster}, Proc. Nat. Acad. Sci. U.S.A. \textbf{83} (1986), no.~10,
  3068--3071. \MR{843307}

\bibitem[Bry90]{MR1071369}
Jean-Luc Brylinski, \emph{Representations of loop groups, {D}irac operators on
  loop space, and modular forms}, Topology \textbf{29} (1990), no.~4, 461--480.
  \MR{1071369}

\bibitem[DFHH14]{MR3223024}
Christopher~L. Douglas, John Francis, Andr\'{e}~G. Henriques, and Michael~A.
  Hill (eds.), \emph{Topological modular forms}, Mathematical Surveys and
  Monographs, vol. 201, American Mathematical Society, Providence, RI, 2014.
  \MR{3223024}

\bibitem[DLMZ04]{MR2106231}
Chongying Dong, Kefeng Liu, Xiaonan Ma, and Jian Zhou, \emph{{$K$}-theory
  associated to vertex operator algebras}, Math. Res. Lett. \textbf{11} (2004),
  no.~5-6, 629--647. \MR{2106231}

\bibitem[DM04]{MR2097833}
Chongying Dong and Geoffrey Mason, \emph{Rational vertex operator algebras and
  the effective central charge}, Int. Math. Res. Not. (2004), no.~56,
  2989--3008. \MR{2097833}

\bibitem[Don93]{MR1245855}
Chongying Dong, \emph{Vertex algebras associated with even lattices}, J.
  Algebra \textbf{161} (1993), no.~1, 245--265. \MR{1245855}

\bibitem[Ebe13]{MR2977354}
Wolfgang Ebeling, \emph{Lattices and codes}, third ed., Advanced Lectures in
  Mathematics, Springer Spektrum, Wiesbaden, 2013, A course partially based on
  lectures by Friedrich Hirzebruch. \MR{2977354}

\bibitem[EG17]{EpaGanter17}
Narthana Epa and Nora Ganter, \emph{Platonic and alternating 2-groups}, High.
  Struct. \textbf{1} (2017), no.~1, 122--146. \MR{3912053}

\bibitem[FLM88]{MR996026}
Igor Frenkel, James Lepowsky, and Arne Meurman, \emph{Vertex operator algebras
  and the {M}onster}, Pure and Applied Mathematics, vol. 134, Academic Press,
  Inc., Boston, MA, 1988. \MR{996026}

\bibitem[Fra92]{MR1235295}
Jens Franke, \emph{On the construction of elliptic cohomology}, Math. Nachr.
  \textbf{158} (1992), 43--65. \MR{1235295}

\bibitem[Fre02]{MR1886764}
Edward Frenkel, \emph{Vertex algebras and algebraic curves}, no. 276, 2002,
  S\'{e}minaire Bourbaki, Vol. 1999/2000, pp.~299--339. \MR{1886764}

\bibitem[GMS00]{MR1748287}
Vassily Gorbounov, Fyodor Malikov, and Vadim Schechtman, \emph{Gerbes of chiral
  differential operators}, Math. Res. Lett. \textbf{7} (2000), no.~1, 55--66.
  \MR{1748287}

\bibitem[HBJ92]{MR1189136}
Friedrich Hirzebruch, Thomas Berger, and Rainer Jung, \emph{Manifolds and
  modular forms}, Aspects of Mathematics, E20, Friedr. Vieweg \& Sohn,
  Braunschweig, 1992, With appendices by Nils-Peter Skoruppa and by Paul Baum.
  \MR{1189136}

\bibitem[Hir77]{MR0480355}
F.~Hirzebruch, \emph{The ring of {H}ilbert modular forms for real quadratic
  fields in small discriminant}, Modular functions of one variable, {VI}
  ({P}roc. {S}econd {I}nternat. {C}onf., {U}niv. {B}onn, {B}onn, 1976), 1977,
  pp.~287--323. Lecture Notes in Math., Vol. 627. \MR{0480355}

\bibitem[Hir88]{MR981372}
Friedrich Hirzebruch, \emph{Elliptic genera of level {$N$} for complex
  manifolds}, Differential geometrical methods in theoretical physics ({C}omo,
  1987), NATO Adv. Sci. Inst. Ser. C Math. Phys. Sci., vol. 250, Kluwer Acad.
  Publ., Dordrecht, 1988, pp.~37--63. \MR{981372}

\bibitem[HL16]{MR3455154}
Michael Hill and Tyler Lawson, \emph{Topological modular forms with level
  structure}, Invent. Math. \textbf{203} (2016), no.~2, 359--416. \MR{3455154}

\bibitem[Hop02]{MR1989190}
M.~J. Hopkins, \emph{Algebraic topology and modular forms}, Proceedings of the
  {I}nternational {C}ongress of {M}athematicians, {V}ol. {I} ({B}eijing, 2002),
  Higher Ed. Press, Beijing, 2002, pp.~291--317. \MR{1989190}

\bibitem[JFT20]{Johnson-FreydTreumann20}
Theo Johnson-Freyd and David Treumann, \emph{{${\rm H}^4({\rm Co}_0; {\bf Z}) =
  {\bf Z}/24$}}, Int. Math. Res. Not. IMRN (2020), no.~21, 7873--7907.
  \MR{4176841}

\bibitem[Kac98]{MR1651389}
Victor Kac, \emph{Vertex algebras for beginners}, second ed., University
  Lecture Series, vol.~10, American Mathematical Society, Providence, RI, 1998.
  \MR{1651389}

\bibitem[Lau99]{MR1660325}
Gerd Laures, \emph{The topological {$q$}-expansion principle}, Topology
  \textbf{38} (1999), no.~2, 387--425. \MR{1660325}

\bibitem[Lau00]{MR1781277}
\bysame, \emph{On cobordism of manifolds with corners}, Trans. Amer. Math. Soc.
  \textbf{352} (2000), no.~12, 5667--5688. \MR{1781277}

\bibitem[Lau04]{MR2076927}
\bysame, \emph{{$K(1)$}-local topological modular forms}, Invent. Math.
  \textbf{157} (2004), no.~2, 371--403. \MR{2076927}

\bibitem[LL07]{MR2294219}
Bong~H. Lian and Andrew~R. Linshaw, \emph{Chiral equivariant cohomology. {I}},
  Adv. Math. \textbf{209} (2007), no.~1, 99--161. \MR{2294219}

\bibitem[LM89]{LawsonMichelsohn89}
H.~Blaine Lawson, Jr. and Marie-Louise Michelsohn, \emph{Spin geometry},
  Princeton Mathematical Series, vol.~38, Princeton University Press,
  Princeton, NJ, 1989. \MR{1031992}

\bibitem[LRS95]{MR1320998}
Peter~S. Landweber, Douglas~C. Ravenel, and Robert~E. Stong, \emph{Periodic
  cohomology theories defined by elliptic curves}, The \v{C}ech centennial
  ({B}oston, {MA}, 1993), Contemp. Math., vol. 181, Amer. Math. Soc.,
  Providence, RI, 1995, pp.~317--337. \MR{1320998}

\bibitem[Lur09]{MR2597740}
J.~Lurie, \emph{A survey of elliptic cohomology}, Algebraic topology, Abel
  Symp., vol.~4, Springer, Berlin, 2009, pp.~219--277. \MR{2597740}

\bibitem[MMS72]{MR398664}
F.~Jessie MacWilliams, Colin~L. Mallows, and Neil J.~A. Sloane,
  \emph{Generalizations of {G}leason's theorem on weight enumerators of
  self-dual codes}, IEEE Trans. Inform. Theory \textbf{IT-18} (1972), 794--805.
  \MR{398664}

\bibitem[MR09]{MR2508904}
Mark Mahowald and Charles Rezk, \emph{Topological modular forms of level 3},
  Pure Appl. Math. Q. \textbf{5} (2009), no.~2, Special Issue: In honor of
  Friedrich Hirzebruch. Part 1, 853--872. \MR{2508904}

\bibitem[MS74]{MR338209}
A.~W. Mason and W.~W. Stothers, \emph{On subgroups of {${\rm GL}(n,A)$} which
  are generated by commutators}, Invent. Math. \textbf{23} (1974), 327--346.
  \MR{338209}

\bibitem[MT91]{MimuraToda91}
Mamoru Mimura and Hirosi Toda, \emph{Topology of {L}ie groups. {I}, {II}},
  Translations of Mathematical Monographs, vol.~91, American Mathematical
  Society, Providence, RI, 1991, Translated from the 1978 Japanese edition by
  the authors. \MR{1122592}

\bibitem[MT10]{MR2648364}
Geoffrey Mason and Michael Tuite, \emph{Vertex operators and modular forms}, A
  window into zeta and modular physics, Math. Sci. Res. Inst. Publ., vol.~57,
  Cambridge Univ. Press, Cambridge, 2010, pp.~183--278. \MR{2648364}

\bibitem[PS86]{PressleySegal86}
Andrew Pressley and Graeme Segal, \emph{Loop groups}, Oxford Mathematical
  Monographs, The Clarendon Press, Oxford University Press, New York, 1986,
  Oxford Science Publications. \MR{900587}

\bibitem[Rom05]{Romagny05}
Matthieu Romagny, \emph{The fundamental theorem of alternating functions},
  \url{https://perso.univ-rennes1.fr/matthieu.romagny/notes/FTAF.pdf},
  September 2005.

\bibitem[Ros94]{MR1282290}
Jonathan Rosenberg, \emph{Algebraic {$K$}-theory and its applications},
  Graduate Texts in Mathematics, vol. 147, Springer-Verlag, New York, 1994.
  \MR{1282290}

\bibitem[Seg88]{MR981378}
G.~B. Segal, \emph{The definition of conformal field theory}, Differential
  geometrical methods in theoretical physics ({C}omo, 1987), NATO Adv. Sci.
  Inst. Ser. C Math. Phys. Sci., vol. 250, Kluwer Acad. Publ., Dordrecht, 1988,
  pp.~165--171. \MR{981378}

\bibitem[SP11]{Schommer-Pries11}
Christopher~J. Schommer-Pries, \emph{Central extensions of smooth 2-groups and
  a finite-dimensional string 2-group}, Geom. Topol. \textbf{15} (2011), no.~2,
  609--676. \MR{2800361}

\bibitem[ST04]{MR2079378}
Stephan Stolz and Peter Teichner, \emph{What is an elliptic object?}, Topology,
  geometry and quantum field theory, London Math. Soc. Lecture Note Ser., vol.
  308, Cambridge Univ. Press, Cambridge, 2004, pp.~247--343. \MR{2079378}

\bibitem[Ste92]{Stembridge92}
John~R. Stembridge, \emph{The projective representations of the hyperoctahedral
  group}, J. Algebra \textbf{145} (1992), no.~2, 396--453. \MR{1144940}

\bibitem[Tam99]{MR1716689}
Hirotaka Tamanoi, \emph{Elliptic genera and vertex operator super-algebras},
  Lecture Notes in Mathematics, vol. 1704, Springer-Verlag, Berlin, 1999.
  \MR{1716689}

\bibitem[Woo89]{MR1022693}
Jay~A. Wood, \emph{Maximal abelian subgroups of spinor groups and
  error-correcting codes}, Algebraic topology ({E}vanston, {IL}, 1988),
  Contemp. Math., vol.~96, Amer. Math. Soc., Providence, RI, 1989,
  pp.~333--350. \MR{1022693}

\end{thebibliography}
\end{document}